\documentclass[11pt,twoside,english]{article}
\usepackage{psfrag,babel}
\usepackage{latexsym}
\usepackage{amsmath,amsthm}
\usepackage{amsfonts}
\usepackage{amsgen}
\usepackage{amssymb}
\usepackage{amstext}
\usepackage{amssymb}
\usepackage{graphicx}
\usepackage{color}
\usepackage{epsfig}
\usepackage[colorinlistoftodos]{todonotes}
\usepackage[modulo,right]{lineno}

\usepackage{dsfont}

\usepackage{tikz}
\usetikzlibrary{decorations.pathreplacing}

\usepackage{hyperref}
\hypersetup{colorlinks=true, citecolor=blue}

\newtheorem{Proposition}{Proposition}[section]
\newtheorem{example}{Example}[section]
\newtheorem{lemma}[Proposition]{Lemma}
\newtheorem{definition}[Proposition]{Definition}

\newtheorem{remark} [Proposition]{Remark}

\newtheorem{theorem}[Proposition]{Theorem}
\newtheorem{corollary}[Proposition]{Corollary}

\usepackage{tikz}
\usetikzlibrary{decorations.pathreplacing}
\usepackage{caption}

\DeclareMathOperator{\Int}{int}
\DeclareMathOperator{\comp}{comp}
\DeclareMathOperator{\id}{id}
\DeclareMathOperator{\dist}{dist}

\DeclareMathOperator{\Crit}{Crit}
\newcommand{\Lim}[1]{\raisebox{0.5ex}{\scalebox{0.8}{$\displaystyle \lim_{#1}\;$}}}

\begin{document}

\title{Robust and generic properties for piecewise continuous maps on the interval}

\author{A. Calder\'{o}n\footnote{Departamento de Ciencias Matem\'aticas y F\'isicas, Facultad de Ingenier\'ia, Universidad Cat\'olica de Temuco, Rudecindo Ortega St. \#2950, Temuco, Chile. Email: {\tt acalderon@uct.cl}}}

\maketitle

\begin{abstract}
We construct an appropriate metric on the collection of piecewise $\mathcal C^r$ maps defined on a compact interval. Although this metric space turns out to be not complete, we show that it is indeed a Baire space. As an application, we prove the robustness of non-degenerate critical points for this class of systems and we show the genericity of piecewise Lipschitz maps that admit an invariant Borel probability measure.\\

\end{abstract}

\section{Introduction}\label{sec0}

Given $r\geq0$, we consider piecewise $\mathcal C^r$ maps (in short PC${}^r$-maps) from a compact interval into itself. We are interested in generical and robust properties associated with this class of dynamical systems. For this, it is necessary to endow the functions space with an appropriate metric, inducing a concept of perturbation that allows the description of dynamic phenomenologies. We consider only maps such that these and their derivatives are uniformly continuous on the interior of each subinterval induced by a finite partition of the definition interval.

In \cite{CB11}, Catsigeras and Budelli perturbed injective piecewise contracting maps on $\mathbb R^n$ to show that their asymptotic dynamics is (pseudo) generically periodic. It is remarkable that the perturbations they define are not reduced to a finite number of real parameters as is often the case in the theory of piecewise continuous maps. However, these perturbations are not induced by a metric; in fact, the neighborhoods they consider are not even a basis for a topology. In dimesion one, A. Nogueira, B. Pires and R. Rosales introduced a concept of Lebesgue-genericity to describe the typical asymptotic dynamics of piecewise contracting maps (see \cite{NPR14,NPR18}). All the known genericity concepts for PC${}^r$-maps involve the Lebesgue measure in some way (see also \cite[Remark 1.4.6]{K98}) and the lack of a useful metrizable topology prevents us to consider topological genericity in this context.

A dynamically useful concept of perturbation is known for continuous maps (respectively, maps of class $\mathcal C^r$) defined on a compact space, where the uniform distance (respectively, $\mathcal C^r$-metric: {\em the one that uniformly compares the map and its derivatives}) works very well.
One of the goal of this paper is the construction of a metric --which we will denote by \underline{$\comp^r$}-- on collection of PC${}^r$-maps capable of describing general dynamic aspects. Among the properties we are going to prove for this {\em PC-metric}, we highlight the following:
\begin{itemize}
    \item We can study perturbations of maps that do not belong to parameterized families;
    \item The proximity between maps does not depend ``too much'' on the discontinuities they have, this idea will allow us to compare maps with different sets of discontinuities, and at the same time, to induce a reasonable concept of convergence (in Section \ref{sec1} this idea is explained in detail);
    \item The functions metric space is a Baire space. This result will allow us to handle the generical properties.
\end{itemize}
We apply this to show that the non-degenerate critical points prevail under small perturbations and that generically the piecewise Lipschitz maps admit an invariant Borel probability measure. Specifically, if $\mathcal{PC}_N^r(X)$ denotes the space of PC${}^r$-maps with $N$ continuity branches and defined on the compact interval $X$, and also $\comp^r$ denotes the {\em piecewise comparison metric} (see Definition \ref{METRICS} in Section \ref{sec2}), the main results of this paper are the following:

\begin{theorem}
$\big(\mathcal{PC}_N^r(X),\comp^r\!\big)$ is a Baire space.
\end{theorem}

\begin{theorem}
Over the space $\big(\mathcal{PC}_N^r(X),\comp^r\!\big)$ the following applications are valid: \\[1ex]
1.$\;$ Suppose that $f\in\mathcal{PC}_N^r(X)$, $r\geq2$, has a finite number of critical points (we will write $\#\Crit(f)<\infty$) and let $\Delta$ be the set of points that are either discontinuities of $f$ or a border of the interval $X$. If each critical point of $f$ is non-degenerate and each one-sided limit of the derivative $f'$ on $\Delta$ is different than 0, then there exists $\epsilon>0$ such that $\#\Crit(g)=\#\Crit(f)$ for every $g\in\mathcal{PC}_N^r(X)$ satisfying $\comp^r(f,g)<\epsilon$.\\[1ex]
2.$\;$ Consider the subspace $\mathcal{P\!L}_N(X)\subset\mathcal{PC}_N^0(X)$ of maps that are piecewise Lipschitz. Then, there exists a residual subset $\mathcal F\subset \mathcal{P\!L}_N(x)$ such that every map of $\mathcal F$ admits an invariant Borel probability measure. 
\end{theorem}

\noindent{\em Organization of the article:} Section \ref{sec1} is dedicated to defining the objects of study and setting the notation that we will use throughout the paper. Here we also mention why we do not use the uniform metric. In Section \ref{sec2}, we introduce the metric $\comp^r$. The idea is to apply a piecewise linear transformation that deforms the continuity pieces from one map into those of another, thus allowing both maps to be uniformly compared. In Section \ref{sec3} we characterize the convergent Cauchy sequences (Theorem \ref{COLLAPSE}): a Cauchy sequence is convergent if and only the number of pieces of continuity is kept at infinity. In Section \ref{sec4} we prove that it is a Baire space using the characterization mentioned previously (Theorem \ref{BAIRE}). Sections \ref{sec5} and \ref{sec6} are devoted to applications of our constructions. We prove the robustness of non-degenerate critical points (Corollary \ref{PC.CRIT.POINT}) and the genericity of piecewise Lipschitz maps that admit invariant Borel probability measures (Corollary \ref{PC.GENE}).

\section{Preliminary definitions}\label{sec1}

Let $X\subset\mathbb{R}$ be a compact interval with non-empty interior and $r\geq0$. A map $f:X\to X$ is said {\em piecewise $\mathcal C^r$-continuous interval map} (hereafter PC-map) if there exists a finite collection of $N\geq2$ open subsets $X_1(f),X_2(f),\ldots,X_N(f)\subset X$, the {\em continuity pieces} of $f$, such that 
\begin{enumerate}
    \item $X_{i_1}(f)\cap X_{i_2}(f)=\emptyset\,$ for all $i_1,i_2\in\{1,\ldots,N\}$ such that $i_1\neq i_2\,$;
    \item $X=\overline{X_1(f)}\cup \overline{X_2(f)}\cup\ldots\cup\overline{X_N(f)}\,$;
    \item $f$ is of class $\mathcal C^r$ on each continuity piece;
    \item For all $j\in\{0,\ldots,r\}$ the map $f^{(j)}$ is uniformly continuous on each continuity piece. Here, $f^{(0)}:=f$ and $f^{(j)}$ denotes the $j$th derivative of the restriction of $f$ to $X_1(f)\cup\ldots\cup X_N(f)$.
\end{enumerate}
\vspace*{1ex}
{\bf Important:} The notation $f^{(j)}$ refers to the $j$th derivative of $f$ defined on $X_1(f)\cup\ldots\cup X_N(f)$, while $f^j:=f\circ\ldots\circ f$ ($j$ times) denotes the $j$th iteration of $f$.
\vspace*{1ex}

For a PC-map $f:X\to X$, we let $c_0,c_N$ denote the extreme points of $X$ and $\Delta_f:=\big\{c_1^f<c_2^f<\cdots<c_{N-1}^f\big\}$ denote the set of the boundaries of the continuity pieces of $f$. That is,
\[
X_1(f)=\big[c_0,c_1^f\big)\,,\; X_2(f)=\big(c_1^f,c_2^f\big)\,,\,\ldots\,,\; X_N(f)=\big(c_{N-1}^f,c_N\big].
\]
Note that the points of $\Delta_f$ are either removable (i.e. continuity points) or jump discontinuities. From the uniform continuity of the $j$th derivative $f^{(j)}$, $1\leq j\leq r$, for any $i\in\{1,\dots,N\}$ the map $f^{(j)}|_{X_i}$ admits a unique continuous extension $f_i^{(j)}=[f^{(j)}]_i:\overline{X_i}\to \mathbb R$. The same happens when $j=0$, but in this case the continuous extension of $f|_{X_i}$ is denoted by $f_i=f_i^{(0)}:\overline{X_i}\to X$.

For $r\geq0$, let $PC^r_N(X)$ be the collection of all piecewise $\mathcal C^r$-continuous interval maps defined on $X$ with $N$ continuity pieces. Define the equivalence relation $\sim$ on $PC^r_N(X)$ given by
\[
f\sim g\qquad\Leftrightarrow\qquad \left\{\!\!\begin{array}{l}
     \Delta_f=\Delta_g\;\text{ and}  \\[1ex]
     \text{$f=g\;$ on $X\setminus\Delta_f$.}
\end{array} \right.
\]
We denote by $\mathcal{PC}^r_N(X)$ the set of equivalence classes $PC^r_N(X)/\!\!\sim$. Note that any two maps in a same class of $\mathcal{PC}^r_N(X)$ have exactly the same continuity pieces and the same continuity extensions. For this reason, we continue to use the notation of piecewise map ($f$, $g$, $h$, etc...) to refer to the classes of $\mathcal{PC}^r_N(X)$. 

Perhaps, the most natural choice for metric on the collection $\mathcal{PC}^r_N(X)$ is the (uniform) $\mathcal C^r$-metric. Given $f,g\in\mathcal{PC}^r_N(X)$, we define the {\em $\mathcal C^r$-uniform distance} between $f$ and $g$ by
\begin{equation}
\dist_\infty^r(f,g):=\sum_{i=1}^{N-1}|c_i^f-c_i^g|+\sum_{j=0}^r\,\sup\big\{|f^{(j)}(x)-g^{(j)}(x)|\ :\ x\in X\setminus(\Delta_f\cup\Delta_g)\big\}.
\label{UNI.DISTANCE}
\end{equation}
It is easy to verify that $f=g$ if and only if $\dist_\infty^r(f,g)=0$, so that $\dist^r_\infty$ is a metric on  $\mathcal{PC}_N^r(X)$. If we fix $\Delta\subset \Int(X)$ with $\#\Delta=N-1$ and we consider the subspace $\mathcal F_\Delta\subset\mathcal{PC}_N^r(X)$ such that $\Delta_f=\Delta$ if and only if $f\in\mathcal F_\Delta$, the metric $\dist_\infty^r$ behaves well on $\mathcal F_{\Delta}$: the pair $\big(\mathcal F_\Delta,\dist^r_\infty\big)$ is a complete metric space and, therefore, is a Baire space. Moreover, on each continuity piece determined by $\Delta$, the proximity between maps is similar to the continuous case. For example, in \cite{BO18} the authors define residual sets of $(\mathcal F_\Delta,\dist^0_\infty)$ to determine certain typical dynamic properties in a topological sense. 

In general, $\dist^r_\infty$ does not behave well in presence of maps with different borders of continuity pieces. Precisely, in $\big(\mathcal{PC}_N^r(X),\dist_\infty^r\big)$ the following phenomenon occurs: maps with ``similar graphs'' are not necessarily close. This phenomenon has consequences on the type of convergence induced by the metric. To explain this idea, consider the piecewise linear maps $f,f_n\in\mathcal{PC}_2^r\big([0,1]\big)$ ($n\geq1$ and $r\geq0$) given by
\[
f_n(x)=\left\{\!\begin{array}{ll}
     \frac{x}{2}+\frac{n+4}{2(n+2)}& \quad\text{ if $x\in\big[0,\frac{n}{2(n+2)}\big)$,} \\[1ex]
     \frac{x}{2} -\frac{n}{4(n+2)}& \quad\text{ if $x\in\big(\frac{n}{2(n+2)},1\big]$}
\end{array}\right.
\qquad\text{and}\qquad f(x)=\left\{\!\begin{array}{ll}
     \frac{x}{2}+\frac{1}{2}& \quad\text{ if $x\in\big[0,\frac{1}{2}\big)$,} \\[1ex]
     \frac{x}{2} - \frac{1}{4}& \quad\text{ if $x\in\big(\frac{1}{2},1\big]$.}
\end{array}\right. 
\]
Clearly, $c_1^{f_n}:=\frac{n}{2(n+2)}<c_1^{f}:=\frac{1}{2}$ for every $n\geq1$. Moreover, for every $\epsilon>0$ there exists $n\geq1$ such that
\[
\big\|(f_n-f)|_{[0,c_1^{f_n})}\big\|_{\infty}<\epsilon\qquad\text{and}\qquad \big\|(f_n-f)|_{(c_1^{f},1]}\big\|_{\infty}<\epsilon.
\]
For large values of $n$, the graphs of $f_n$ and $f$ are similar (the graph of $f_n$ converges to the graph of $f$ in the Hausdorff topology); however, $f_n$ does not converge to $f$ w.r.t. $\dist_\infty^r$. In fact, note that $\dist_\infty^r(f_n,f)\geq\frac{1}{2}$ for all $n\geq1$.
\begin{center}
\noindent\begin{minipage}[c]{.4\linewidth}
  \centering
    \tikzset{every picture/.style={line width=0.75pt}} 
\begin{tikzpicture}[x=0.75pt,y=0.75pt,yscale=-1,xscale=1]
\draw    (169.33,250.22) -- (169.33,236.67) -- (169.33,53.17) ;
\draw [shift={(169.33,50.17)}, rotate = 450] [fill={rgb, 255:red, 0; green, 0; blue, 0 }  ][line width=0.08]  [draw opacity=0] (8.93,-4.29) -- (0,0) -- (8.93,4.29) -- cycle    ;
\draw    (169.33,250.22) -- (395.11,250.22) ;
\draw [shift={(398.11,250.22)}, rotate = 180] [fill={rgb, 255:red, 0; green, 0; blue, 0 }  ][line width=0.08]  [draw opacity=0] (8.93,-4.29) -- (0,0) -- (8.93,4.29) -- cycle    ;
\draw [color={rgb, 255:red, 74; green, 144; blue, 226 }  ,draw opacity=0.7 ][line width=0.75]  [dash pattern={on 4.5pt off 4.5pt}]  (300,70.22) -- (300,250.25) ;
\draw  [fill={rgb, 255:red, 74; green, 144; blue, 226 }  ,fill opacity=0.78 ] (298,251.67) .. controls (298.78,252.65) and (300.21,252.82) .. (301.2,252.04) .. controls (302.19,251.26) and (302.35,249.83) .. (301.57,248.84) .. controls (300.79,247.86) and (299.36,247.69) .. (298.37,248.47) .. controls (297.39,249.25) and (297.22,250.68) .. (298,251.67) -- cycle ;
\draw  [color={rgb, 255:red, 74; green, 144; blue, 226 }  ,draw opacity=0.34 ] (169.83,252.17) .. controls (169.83,256.84) and (172.16,259.17) .. (176.83,259.17) -- (233.33,259.17) .. controls (240,259.17) and (243.33,261.5) .. (243.33,266.17) .. controls (243.33,261.5) and (246.66,259.17) .. (253.33,259.17)(250.33,259.17) -- (292.33,259.17) .. controls (297,259.17) and (299.33,256.84) .. (299.33,252.17) ;
\draw  [color={rgb, 255:red, 74; green, 144; blue, 226 }  ,draw opacity=0.34 ] (300.83,252.5) .. controls (300.83,257.17) and (303.16,259.5) .. (307.83,259.5) -- (352.44,259.5) .. controls (359.11,259.5) and (362.44,261.83) .. (362.44,266.5) .. controls (362.44,261.83) and (365.77,259.5) .. (372.44,259.5)(369.44,259.5) -- (372.44,259.5) .. controls (377.11,259.5) and (379.44,257.17) .. (379.44,252.5) ;
\draw  [fill={rgb, 255:red, 208; green, 2; blue, 27 }  ,fill opacity=0.7 ] (276,251.33) .. controls (276.78,252.32) and (278.21,252.49) .. (279.2,251.71) .. controls (280.19,250.93) and (280.35,249.5) .. (279.57,248.51) .. controls (278.79,247.52) and (277.36,247.35) .. (276.37,248.13) .. controls (275.39,248.91) and (275.22,250.35) .. (276,251.33) -- cycle ;
\draw [color={rgb, 255:red, 208; green, 2; blue, 27 }  ,draw opacity=0.78 ][fill={rgb, 255:red, 208; green, 2; blue, 27 }  ,fill opacity=0.7 ][line width=0.75]  [dash pattern={on 4.5pt off 4.5pt}]  (277.79,69.89) -- (277.79,249.92) ;
\draw  [color={rgb, 255:red, 208; green, 2; blue, 27 }  ,draw opacity=0.36 ] (278.33,251) .. controls (278.33,255.67) and (280.66,258) .. (285.33,258) -- (317.83,258) .. controls (324.5,258) and (327.83,260.33) .. (327.83,265) .. controls (327.83,260.33) and (331.16,258) .. (337.83,258)(334.83,258) -- (372.33,258) .. controls (377,258) and (379.33,255.67) .. (379.33,251) ;
\draw [color={rgb, 255:red, 208; green, 2; blue, 27 }  ,draw opacity=1 ][line width=1.5]    (168.83,101.25) -- (277.79,70.89) ;
\draw  [color={rgb, 255:red, 208; green, 2; blue, 27 }  ,draw opacity=0.36 ] (169.83,250.5) .. controls (169.83,255.17) and (172.16,257.5) .. (176.83,257.5) -- (186.83,257.5) .. controls (193.5,257.5) and (196.83,259.83) .. (196.83,264.5) .. controls (196.83,259.83) and (200.16,257.5) .. (206.83,257.5)(203.83,257.5) -- (270.33,257.5) .. controls (275,257.5) and (277.33,255.17) .. (277.33,250.5) ;
\draw [color={rgb, 255:red, 208; green, 2; blue, 27 }  ,draw opacity=1 ][line width=1.5]    (277.79,249.92) -- (379.83,221.92) ;
\draw [color={rgb, 255:red, 74; green, 144; blue, 226 }  ,draw opacity=1 ][line width=1.5]    (169.33,106.92) -- (300,70.22) ;
\draw [color={rgb, 255:red, 74; green, 144; blue, 226 }  ,draw opacity=1 ][line width=1.5]    (299.79,250.25) -- (379.83,228.42) ;
\draw    (279.5,214.5) -- (298.33,214.5) ;
\draw    (298.33,211.75) -- (298.33,217) ;
\draw    (279.33,211.75) -- (279.33,217) ;
\draw   (302,243.08) .. controls (306.67,243.08) and (309,240.75) .. (309,236.08) -- (309,167.33) .. controls (309,160.66) and (311.33,157.33) .. (316,157.33) .. controls (311.33,157.33) and (309,154) .. (309,147.33)(309,150.33) -- (309,78.58) .. controls (309,73.91) and (306.67,71.58) .. (302,71.58) ;
\draw  [fill={rgb, 255:red, 208; green, 2; blue, 27 }  ,fill opacity=1 ] (276,251.33) .. controls (276.78,252.32) and (278.21,252.49) .. (279.2,251.71) .. controls (280.19,250.93) and (280.35,249.5) .. (279.57,248.51) .. controls (278.79,247.52) and (277.36,247.35) .. (276.37,248.13) .. controls (275.39,248.91) and (275.22,250.35) .. (276,251.33) -- cycle ;
\draw  [fill={rgb, 255:red, 74; green, 144; blue, 226 }  ,fill opacity=1 ] (298.21,251.67) .. controls (298.99,252.65) and (300.43,252.82) .. (301.41,252.04) .. controls (302.4,251.26) and (302.57,249.83) .. (301.79,248.84) .. controls (301.01,247.86) and (299.57,247.69) .. (298.59,248.47) .. controls (297.6,249.25) and (297.43,250.68) .. (298.21,251.67) -- cycle ;
\draw  [dash pattern={on 0.84pt off 2.51pt}]  (169.33,70.89) -- (380,70.89) ;
\draw  [dash pattern={on 0.84pt off 2.51pt}]  (380,70.89) -- (380,250.22) ;
\draw (245.5,272.67) node  {$\scriptstyle X_{1}(f)$};
\draw (371.5,273.5) node  {$\scriptstyle X_{2}(f)$};
\draw (199.5,272.17) node  {$\scriptstyle X_{1}(f_n)$};
\draw (328,274.17) node   {$\scriptstyle X_{2}(f_n)$};
\draw (288,206) node {$\scriptstyle \frac{1}{n+2}$};
\draw (256,91.67) node  {$f$};
\draw (185,86.67) node  {$f_n$};
\draw (155,72) node  {$1$};
\draw (155,250) node  {$0$};
\draw (301,263.67) node    {$\scriptstyle c^{f}_{1}$};
\draw (277.5,264.17) node  {$\scriptstyle c^{f_n}_{1}$};
\draw (330,156.83) node  [font=\scriptsize] {$\geq\frac{1}{2}$};
\end{tikzpicture}
\end{minipage}\hspace*{0ex}\begin{minipage}[b]{.6\linewidth}
\makeatletter
\def\@captype{figure}
\makeatother
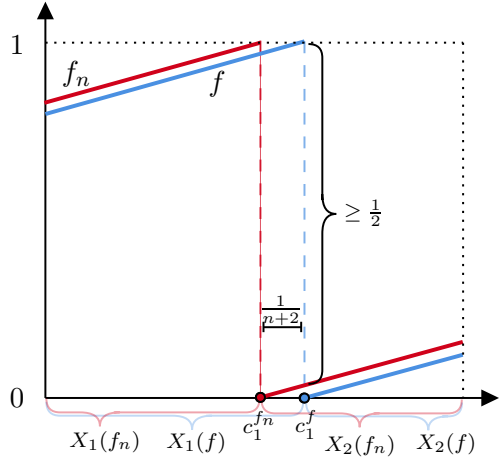
\caption{\small Graphs of $f$ (blue) and $f_n$, $n\geq1$ (red). \\[2.2cm]
}
\vspace*{-.5cm}
\label{fig:SIMILAR.MAPS}
\end{minipage}
\end{center}
This example shows that, w.r.t. $\dist^r_\infty$, the discontinuities have a strong influence in the convergence. Thus, we have that two maps $f,g\in\mathcal{PC}^r_2(X)$ are close (according to $\dist_\infty^r$) if $\Delta_g=\{c_1^f\}$ and $f$ and $g$ are uniformly close (in a certain sense, $g$ belong to a {\em vertical neighborhood} of $f$).\\[-2ex]

\noindent To conclude our analysis on $\dist_\infty^r$, we show that $\big(\mathcal{PC}_N^r(X),\dist_\infty^r\big)$ is not complete.

\begin{example}\em\label{EXAMPLE2.CAUCHY}
Consider the sequence $\big\{f_n:[0,1]\to[0,1]\big\}_{n\in\mathbb N}\subset\mathcal{PC}_2^0\big([0,1]\big)$ defined for any $n\in\mathbb N$ by
\[
f_n(x):=\left\{\begin{array}{ll}
    \frac{x}{2^n} & \quad \text{if $\,x\in\big[0,\frac{1}{2^{n+1}}\big)$,} \\[1ex]
    0 & \quad \text{if $\,x\in \big(\frac{1}{2^{n+1}},1\big]$.}
\end{array}\right.
\]
For any $n,m\geq0$ such that $n>m$, we have that
\setlength \arraycolsep{2pt}
\begin{eqnarray*}
\dist^0_\infty(f_n,f_m)&=&\Big|\frac{1}{2^{n+1}}-\frac{1}{2^{m+1}}\Big|+\sup\big\{|f_n(x)-f_m(x)|\ :\ x\in X\setminus\{1/2^{n+1},1/2^{m+1}\}\big\}\\
&\leq&\frac{1}{2^{n+1}}+\frac{1}{2^{m+1}}+\frac{1}{2^{2m+1}},
\end{eqnarray*}
which implies that $\{f_n\}_{n\in\mathbb N}$ is a Cauchy sequence w.r.t. the metric $\dist^0_\infty$. However, the limit corresponds to a map such that one of its pieces of continuity has an empty interior, which does not belong to $\mathcal{PC}^0_2\big([0,1]\big)$.
\end{example}

\section{The piecewise comparison metric}\label{sec2}

In this section we consider a (piecewise) linear deformation of the continuity pieces of PC${}^r$-maps. If we also consider the Hausdorff distance between the pieces of continuity, we obtain a metric on $\mathcal{PC}_N^r(X)$. Furthermore, unlike what happens with the metric $\dist_\infty^r$, we show that the perturbations of discontinuity points (in the graph) are in all directions (see Lemma \ref{PERT}), which gives us a geometric idea of the neighborhoods of this new metric space.

\begin{definition}\em\label{METRICS}
Given $f,g\in \mathcal{PC}^r_N(X)$, we define the {\em $\mathcal C^r$-piecewise comparison} between $f$ and $g$ by
\begin{equation}
\comp^r(f,g):=\sum_{i=1}^{N-1}|c_i^f-c_i^g|+ \sum_{i=1}^N\sum_{j=0}^r\,\max_{x\in\overline{X_i(f)}}\big|f_i ^{(j)}(x)-g_i^{(j)}\circ \xi(f,g;x)\big|,
\label{HF.DISTANCE}
\end{equation}
where $\xi(f,g;\cdot):X\to X$ is the increasing piecewise linear homeomorphism given by 
\[
\xi(f,g;x):=\frac{c_i^g-c_{i-1}^g}{c_i^f-c_{i-1}^f}(x-c_i^f)+c_i^g \qquad\forall\,x\in\overline{X_i(f)}\,,\;\;\forall\,i\in\{1,\ldots,N\}.
\]
\end{definition}
\noindent We point out some properties of the map $\xi$ in Definition \ref{METRICS}. Let $f,g,h\in\mathcal{PC}^r_N(X)$, then\\[-2ex]

\begin{enumerate}
    \item[{\bf (P1)}] For every $i\in\{1,\ldots,N\}$, the segment $\xi_i(f,g;\cdot):=\xi(f,g;\cdot)\big|_{\overline{X_i(f)}}$ maps $\overline{X_i(f)}$ onto $\overline{X_i(g)}$; 
\end{enumerate}

\noindent\begin{minipage}[b]{0.56\linewidth}
\begin{enumerate}
    \item[{\bf (P2)}] The inverse of $\xi(f,g;\cdot)$ is $\xi(g,f;\cdot)$. Moreover, $\xi_i^{-1}(f,g;\cdot)=\xi_i(g,f;\cdot)$ for every $i\in\{1,\ldots,N\}$;
    \item[{\bf (P3)}] $\xi(f,f;\cdot)$ is the identity map on $X$. Moreover, for every $f_1,f_2\in\mathcal{PC}_N^r(X)$ and $i\in\{1,\ldots,N\}$, if $\overline{X_i(f_1)}=\overline{X_i(f_2)}$ then $\xi_i(f_1,f_2;\cdot)$ is the identity map on $\overline{X_i(f_1)}$;
    \item[{\bf (P4)}] $\xi\big(f,g;\xi(h,f;\cdot)\big)=\xi(h,g;\cdot)$. Moreover, for every $i\in\{1,\ldots,N\}$ we have 
    \begin{equation}
        \xi_i\big(f,g;\xi_i(h,f;\cdot)\big)=\xi_i(h,g;\cdot)\,;\label{XIXI}
    \end{equation}
    \item[{\bf (P5)}] If $\epsilon>0$ and $\comp^r(f,g)<\epsilon$, then $|\xi(f,g;x)-x|<\epsilon$ for any $x\in X$;
\end{enumerate}
\end{minipage}\hspace*{4ex}\begin{minipage}[b]{0.4\linewidth}
\makeatletter
\def\@captype{figure}
\makeatother
  \centering
\tikzset{every picture/.style={line width=0.75pt}} 
\begin{tikzpicture}[x=0.75pt,y=0.75pt,yscale=-1,xscale=1,scale=.43]

\draw [color={rgb, 255:red, 0; green, 0; blue, 0 }  ,draw opacity=1 ] [dash pattern={on 0.84pt off 2.51pt}]  (392.5,98) -- (392.56,110.99) -- (393.75,391.33) -- (393.75,421) ;
\draw [color={rgb, 255:red, 0; green, 0; blue, 0 }  ,draw opacity=1 ] [dash pattern={on 0.84pt off 2.51pt}]  (267.5,98) -- (267.56,110.99) -- (268.75,391.33) -- (268.75,421) ;
\draw [color={rgb, 255:red, 0; green, 0; blue, 0 }  ,draw opacity=1 ] [dash pattern={on 0.84pt off 2.51pt}]  (219.5,98) -- (219.56,110.99) -- (220.75,391.33) -- (220.75,421) ;
\draw  [dash pattern={on 0.84pt off 2.51pt}]  (59.75,397) -- (419.5,397) ;
\draw  [dash pattern={on 0.84pt off 2.51pt}]  (59.75,316) -- (419.5,316) ;
\draw  [dash pattern={on 0.84pt off 2.51pt}]  (59.75,187) -- (419.5,187) ;
\draw  [dash pattern={on 0.84pt off 2.51pt}]  (59.75,139) -- (420.5,139) ;
\draw    (60,421) -- (60,73.67) ;
\draw [shift={(60,71.67)}, rotate = 450] [color={rgb, 255:red, 0; green, 0; blue, 0 }  ][line width=0.75]    (10.93,-3.29) .. controls (6.95,-1.4) and (3.31,-0.3) .. (0,0) .. controls (3.31,0.3) and (6.95,1.4) .. (10.93,3.29)   ;
\draw    (60,421) -- (457.5,421) ;
\draw [shift={(459.5,421)}, rotate = 180] [color={rgb, 255:red, 0; green, 0; blue, 0 }  ][line width=0.75]    (10.93,-3.29) .. controls (6.95,-1.4) and (3.31,-0.3) .. (0,0) .. controls (3.31,0.3) and (6.95,1.4) .. (10.93,3.29)   ;
\draw [line width=3]    (60,421) -- (419.5,421) ;
\draw [line width=3]    (60,423) -- (60,97.67) ;
\draw  [dash pattern={on 0.84pt off 2.51pt}]  (60,97.67) -- (419.5,97.67) ;
\draw  [dash pattern={on 0.84pt off 2.51pt}]  (419.5,97.67) -- (419.5,375.67) -- (419.5,421) ;
\draw  [fill={rgb, 255:red, 255; green, 255; blue, 255 }  ,fill opacity=1 ] (112.75,421) .. controls (112.75,418.79) and (114.54,417) .. (116.75,417) .. controls (118.96,417) and (120.75,418.79) .. (120.75,421) .. controls (120.75,423.21) and (118.96,425) .. (116.75,425) .. controls (114.54,425) and (112.75,423.21) .. (112.75,421) -- cycle ;
\draw [color={rgb, 255:red, 0; green, 0; blue, 0 }  ,draw opacity=1 ] [dash pattern={on 0.84pt off 2.51pt}]  (116.5,98) -- (116.5,387.33) -- (116.5,417) ;
\draw    (61.78,419.56) -- (116.5,397) ;
\draw    (116.5,397) -- (220.11,317.22) ;
\draw  [fill={rgb, 255:red, 255; green, 255; blue, 255 }  ,fill opacity=1 ] (216.75,421) .. controls (216.75,418.79) and (218.54,417) .. (220.75,417) .. controls (222.96,417) and (224.75,418.79) .. (224.75,421) .. controls (224.75,423.21) and (222.96,425) .. (220.75,425) .. controls (218.54,425) and (216.75,423.21) .. (216.75,421) -- cycle ;
\draw  [fill={rgb, 255:red, 255; green, 255; blue, 255 }  ,fill opacity=1 ] (264.75,421) .. controls (264.75,418.79) and (266.54,417) .. (268.75,417) .. controls (270.96,417) and (272.75,418.79) .. (272.75,421) .. controls (272.75,423.21) and (270.96,425) .. (268.75,425) .. controls (266.54,425) and (264.75,423.21) .. (264.75,421) -- cycle ;
\draw  [fill={rgb, 255:red, 255; green, 255; blue, 255 }  ,fill opacity=1 ] (389.75,421) .. controls (389.75,418.79) and (391.54,417) .. (393.75,417) .. controls (395.96,417) and (397.75,418.79) .. (397.75,421) .. controls (397.75,423.21) and (395.96,425) .. (393.75,425) .. controls (391.54,425) and (389.75,423.21) .. (389.75,421) -- cycle ;
\draw  [fill={rgb, 255:red, 255; green, 255; blue, 255 }  ,fill opacity=1 ] (55.75,397) .. controls (55.75,394.79) and (57.54,393) .. (59.75,393) .. controls (61.96,393) and (63.75,394.79) .. (63.75,397) .. controls (63.75,399.21) and (61.96,401) .. (59.75,401) .. controls (57.54,401) and (55.75,399.21) .. (55.75,397) -- cycle ;
\draw  [fill={rgb, 255:red, 255; green, 255; blue, 255 }  ,fill opacity=1 ] (55.75,316) .. controls (55.75,313.79) and (57.54,312) .. (59.75,312) .. controls (61.96,312) and (63.75,313.79) .. (63.75,316) .. controls (63.75,318.21) and (61.96,320) .. (59.75,320) .. controls (57.54,320) and (55.75,318.21) .. (55.75,316) -- cycle ;
\draw  [fill={rgb, 255:red, 255; green, 255; blue, 255 }  ,fill opacity=1 ] (55.75,187) .. controls (55.75,184.79) and (57.54,183) .. (59.75,183) .. controls (61.96,183) and (63.75,184.79) .. (63.75,187) .. controls (63.75,189.21) and (61.96,191) .. (59.75,191) .. controls (57.54,191) and (55.75,189.21) .. (55.75,187) -- cycle ;
\draw  [fill={rgb, 255:red, 255; green, 255; blue, 255 }  ,fill opacity=1 ] (55.75,139) .. controls (55.75,136.79) and (57.54,135) .. (59.75,135) .. controls (61.96,135) and (63.75,136.79) .. (63.75,139) .. controls (63.75,141.21) and (61.96,143) .. (59.75,143) .. controls (57.54,143) and (55.75,141.21) .. (55.75,139) -- cycle ;
\draw    (267.5,187.5) -- (391.78,140) ;
\draw    (220.11,317.22) -- (268.5,186.5) ;
\draw    (391.78,140) -- (419.78,97.72) ;

\draw (50,431.2) node [anchor=north west][inner sep=0.75pt]    {\tiny $X_{1}( f)$};
\draw (132,431.2) node [anchor=north west][inner sep=0.75pt]    {\tiny $X_{2}( f)$};
\draw (213,431.2) node [anchor=north west][inner sep=0.75pt]    {\tiny $X_{3}( f)$};
\draw (300,431.4) node [anchor=north west][inner sep=0.75pt]    {\tiny $X_{4}( f)$};
\draw (383,431.4) node [anchor=north west][inner sep=0.75pt]    {\tiny $X_{5}( f)$};
\draw (425,392) node [anchor=north west][inner sep=0.75pt]    {\tiny $X_{1}( g)$};
\draw (425,340.2) node [anchor=north west][inner sep=0.75pt]    {\tiny $X_{2}( g)$};
\draw (425,241.2) node [anchor=north west][inner sep=0.75pt]    {\tiny $X_{3}( g)$};
\draw (425,150.2) node [anchor=north west][inner sep=0.75pt]    {\tiny $X_{4}( g)$};
\draw (425,103.2) node [anchor=north west][inner sep=0.75pt]    {\tiny $X_{5}( g)$};
\draw (279,202.2) node [anchor=north west][inner sep=0.75pt]    {\small $\xi ( f,g;\cdot )$};
\end{tikzpicture}

\vspace*{3ex}
\hspace*{0ex}\begin{minipage}[b]{.8\linewidth}
\vspace*{-1ex}
\caption{\small Example of the ho\-meo\-mor\-phism $\xi(f,g;\cdot)$, where $f,g\in\mathcal{PC}^r_5(X)$.}
\label{fig:MAPXI}
\end{minipage}
\vspace*{-1ex}
\end{minipage}

\vspace{1ex}
\begin{enumerate}
    \item[{\bf (P6)}] Let $i\in\{1,\ldots,N\}$ and $x\in X\setminus\Delta_f$. Then, $x\in X_i(f)$ if and only if $\xi(f,g;x)\in X_i(g)$.
\end{enumerate}
\noindent Using the properties {\bf(P1)}--{\bf(P4)} we can prove that the $\mathcal C^r$-piecewise comparison $\comp^r$ is a metric on $\mathcal{PC}^r_N(X)$. From {\bf (P3)}, it is easy to verify that $f=g$ if and only if $\comp^r(f,g)=0$, and symmetry is consequence of {\bf (P1)}--{\bf (P2)}. Finally, to conclude the triangle ine\-qua\-li\-ty we compute for $f,g,h\in\mathcal{PC}_N^r(X)$:
\setlength \arraycolsep{2pt}
\begin{eqnarray*}
\comp^r(f,g)&\leq&\sum_{i=1}^{N-1}|c_i^f-c_i^h|+ \sum_{i=1}^N\sum_{j=0}^r\,\max_{x\in\overline{X_i(f)}}\left|f_i ^{(j)}(x)-h_i^{(j)}\circ\, \xi_i(f,h;x)\right|\\[-.5ex]
&&\hspace*{.5cm}+\sum_{i=1}^{N-1}|c_i^h-c_i^g|+\sum_{i=1}^N\sum_{j=0}^r\,\max_{x\in\overline{X_i(f)}}\left|h_i^{(j)}\circ\,\xi_i(f,h;x)-g_i^{(j)}\circ\, \xi_i(f,g;x)\right|\\
&=&\comp^r(f,h)+\sum_{i=1}^{N-1}|c_i^h-c_i^g|\\[-.5ex]
&&\hspace*{2.6cm}+\sum_{i=1}^N\sum_{j=0}^r\,\max_{x\in\overline{X_i(h)}}\left|h_i^{(j)}(x)-g_i^{(j)}\circ\, \xi_i\big(f,g;\xi_i(h,f;x)\big)\right|\\
&=&\comp^r(f,h)+\comp^r(h,g),
\end{eqnarray*}
where we used {\bf(P4)}. Thus, we conclude that $\comp^r$ is a metric on $\mathcal{PC}^r_N(X)$. \\[-2ex]

\noindent In previous section we gave an example where we explained why near jump discontinuities the neighborhoods according to $\dist_\infty^r$ are {\em vertical}. A potential notion of stability is affected by the class of perturbations that are only vertical near jump discontinuities. Note that in Figure \ref{fig:SIMILAR.MAPS} we can see that each map $f_n$ is excluded from every (vertical) neighborhood of $f$. In this way, it is not possible to relate the concept of stability given by $\dist_\infty^r$ to the properties shared by the systems $f$ and $f_n$ with $n$ sufficiently large. Thus, we deduce that a metric on $\mathcal{PC}_N^r(X)$ that induces a reasonable stability concept (and a reasonable convergence) in the sense described above should admit perturbations of the concatenated points 
\[
\Big(c,\lim_{x\to c^-} f(x)\Big)\quad\text{and}\quad \Big(c,\lim_{x\to c^+} f(x)\Big)\qquad \forall\,c\in\Delta_f,
\]
in all directions as elements of the graph of $f$, which are contained in $X\times X$. We formalize this idea in the following result (Lemma \ref{PERT}), where we show that the perturbations induced by $\comp^r$ at each point of the graph of a PC-map span all possible directions.

\begin{lemma}\label{PERT}
Let $f\in\mathcal{PC}^0_N(X)$, $i\in\{1,\ldots,N\}$, $\epsilon>0$ (small enough) and $$(a,b)\;\in\;\left(c_i^f-\frac{\epsilon}{N},c_i^f+\frac{\epsilon}{N}\right)\times\left(\Lim{x\to [c_i^f]^-}f(x)-\frac{\epsilon}{N}\,,\Lim{x\to [c_i^f]^-}f(x)+\frac{\epsilon}{N}\right)\!.$$ Then, there exists $g\in\mathcal{PC}^0_N(X)$ such that $c_i^g=a$, $\Lim{x\to a^-}g(x)=b$ and $\comp^r(f,g)<\epsilon$.
\end{lemma}

\begin{proof}
Let $\xi:X\to X$ be the piecewise linear homeomorphism that linearly deforms the continuity pieces induced by $\Delta_f$ into the continuity pieces induced by $\Delta^*:=\big(\Delta_f\setminus\{c_i^f\}\big)\cup\{a\}$, and consider $g\in\mathcal{PC}^r_N(X)$ with $\Delta_g=\Delta^*$ defined by
\[
g(y):=\left\{\begin{array}{ll}
    \displaystyle f\circ\xi^{-1}(y) -\Lim{x\to [c_i^f]^-}f(x)+b &\quad \text{if $y\in X_i(f)\,$;}   \\[1ex]
     f\circ\xi^{-1}(y) & \quad \text{if $y\in X_j(f)$ with $j\neq i$.}
\end{array}\right.
\]
It is easy to verify that $g$ satisfies the required conditions.
\end{proof}




\begin{remark}\em\label{EQUAL.DELTA}
For $\epsilon>0$ and $f,g\in\mathcal{PC}_N^r(X)$ such that $\Delta_f=\Delta_g$, we have
\begin{equation*}
    \dist_\infty^r(f,g)<\epsilon \qquad\Leftrightarrow\qquad \comp^r(f,g)<N\epsilon.
\end{equation*}
Thus, both metrics are equivalent when restricted to subspaces $\mathcal F_\Delta$ of maps that share the same borders of continuity pieces, which are given by some subset $\Delta\subset\Int(X)$ with $\#\Delta=N-1$. Thus, we deduce that the pair $\big(\mathcal F_\Delta,\comp^r\!\big)$ is a complete metric space.
\end{remark}

\noindent Similarly to $\big(\mathcal{PC}_N^r(X),\dist_\infty^r\big)$, the metric space $\big(\mathcal{PC}_N^r(X),\comp^r\!\big)$ is not complete.

\begin{example}\label{NOT.CONVERGENT}\em 
Consider the sequence $\{f_n\}_{n\geq2}\subset\mathcal{PC}^{0}_2\big([0,1]\big)$ given by
\[
f_n(x):=\left\{\begin{array}{ll}
    \frac{x}{2} & \quad \text{if $\,x\in\big[0,\frac{1}{n}\big)$,} \\[1ex]
    \frac{x}{2}+\frac{1}{2} & \quad \text{if $\,x\in \big(\frac{1}{n},1\big]$.}
\end{array}\right.
\]
For each $n,m\geq0$, the segments $\xi_1(f_n,f_m;\cdot)$ and $\xi_2(f_n,f_m;\cdot)$ are given by
\begin{eqnarray*}
\xi_1(f_n,f_m;x_1)=\frac{n}{m}x_1\qquad\text{and}\qquad \xi_2(f_n,f_m;x_2)=\frac{(m-1)n}{(n-1)m}x_2+\frac{n-m}{(n-1)m},
\end{eqnarray*}
for every $x_1\in\big[0,1/(n+2)\big]$ and $x_2\in\big[1/(n+2),1\big]$, respectively. Thus, we have
\setlength \arraycolsep{2pt}
\begin{eqnarray*}
\comp^0(f_n,f_m)&=&\left|\frac{1}{n}-\frac{1}{m}\right|+\max_{x\in[0,1/n]}\left|\frac{x}{2}-\frac{\xi_1(f_n,f_m;x)}{2}\right|+\max_{x\in[1/n,1]}\left|\frac{x}{2}-\frac{\xi_2(f_n,f_m;x)}{2}\right|\\
&\leq& \frac{5}{2(n-1)}+\frac{5}{2m}+\frac{1}{(n-1)m},
\end{eqnarray*}
which proves that $\{f_n\}_{n\geq2}$ is a Cauchy sequence w.r.t. $\comp^0$. However, the limit corresponds to a map such that one of its pieces of continuity has an empty interior, which does not belong to $\mathcal{PC}^0_2\big([0,1]\big)$.
\end{example}

\noindent Note that the sequence $\{f_n\}_{n\geq2}$ of Example \ref{NOT.CONVERGENT} is not a Cauchy sequence w.r.t. $\dist^0_\infty$. In fact, for any $n,m\geq2$ such that $n>m$ we have
\setlength \arraycolsep{2pt}
\begin{eqnarray*}
\dist^0_\infty(f_n,f_m)&=&\left|\frac{1}{n}-\frac{1}{m}\right|+\sup\big\{\big|f_n(x)-f_m(x)\big|\ :\ x\in X\setminus\{1/n,1/m\}\big\}\geq \frac{1}{2}.
\end{eqnarray*}
This allows us to deduce that the topologies induced by $\dist_\infty^r$ and $\comp^r$ are different.

\section{Characterization of convergent Cauchy sequences}\label{sec3}

In the present section we prove Theorem \ref{COLLAPSE}, which is a simple and useful characterization of the convergent Cauchy sequences in the metric space $\big(\mathcal{PC}^r_N(X),\comp^r\big)$. The method used to prove it consists in ``decoupling the problem''; that is, working with a space of continuous functions with variable compact domain. 


Recall that $X$ denotes a compact interval with a non-empty interior. Let $\mathbb I(X)$ be the collection of all compact intervals with non-empty interior contained in $X$, and for every $W\subset\mathbb R$ a compact interval (with non-empty interior) or the entire real line $\mathbb R$, let
\[
\mathbb{FI}(X;W):=\big\{(f,I)\ :\ f:I\to W \,\text{ is a continuous map with}\ I\in\mathbb I(X)\big\}.
\]
For convenience, we understand $\mathbb{FI}(X;W)$ as a space of functions and not as ordered pairs. Define the metric $\chi:\mathbb{FI}(X;W)\times\mathbb{FI}(X;W)\to[0,\infty)$ according to
\[
\chi(f,g):=\dist_{\mathcal H}(I_f,I_g)+\max_{x\in I_f}\big|f(x)-g\circ\xi(f,g;x)\big|\qquad\forall\,f,g\in\mathbb{FI}(X;W),
\]
where $I_f,I_g\in\mathbb I(X)$ denote respectively the intervals of definition of $f$ and $g$, $\xi(f,g;\cdot):I_f\to I_g$ is the increasing map that linearly deforms $I_f$ onto $I_g$ \big(as in the previous section, $\xi$ satisfies properties analogous to {\bf (P1)}--{\bf (P5)}\big) and $\dist_{\mathcal H}$ denotes the Hausdorff metric on $\mathbb I(X)$ given by
\[
\dist_{\mathcal H}(I,J):=\max\Big\{\!\max_{x\in I}\min_{y\in J}|x-y|,\max_{y\in J}\min_{x\in I}|x-y|\Big\}\qquad\forall\,I,J\in\mathbb I(X).
\]
By using an argument similar to the one given in the proof that $\comp^r$ is a metric, we can conclude $\chi$ is a metric on $\mathbb{FI}(X;W)$. If $\mathbb G(X)$ denotes the collection of all compact nonempty subsets of $X$, then $\mathbb G(X)$ endowed with the Hausdorff distance is a complete metric space (see \cite{H99}). Since $\mathbb J(X):=\mathbb I(X)\cup\big\{\{x\}\ :\ x\in X\big\}$ is the closure of $\mathbb I(X)$ in $(\mathbb G(X),\dist_{\mathcal H})$, we have that $(\mathbb J(X),\dist_{\mathcal H})$ defines a complete metric space too. This last result will be required to prove the following:

\begin{lemma}\label{FIX}
Let $\{f_n\}_{n\in\mathbb N}$ be a Cauchy sequence in $\big(\mathbb{FI}(X;W),\chi\big)$ such that $I_{f_n}\to I$ \big(w.r.t. $\dist_{\mathcal H}$\big) for some $I\in\mathbb I(X)$. Then, $f_n\to f$ for some $f\in\mathbb{FI}(X;W)$.
\end{lemma}

\begin{proof}
Since $\{f_n\}_{n\in\mathbb N}$ is a Cauchy sequence, for all $\epsilon>0$ there exists $n_\epsilon\geq0$ such that
\vspace*{-.05cm}

\noindent\begin{minipage}[b]{.72\textwidth}
\begin{eqnarray}
\chi(f_n,f_m)<\epsilon\qquad\forall\,n,m\geq n_\epsilon.\label{CHI.FNFM}
\end{eqnarray}
First, we define the sequence $\{s_n(x)\}_{n\in\mathbb N}$. For every $x\in I_{f_0}$ let $s_0(x):=x$ and
\[
s_n(x):=\xi\big(f_{n-1},f_{n};s_{n-1}(x)\big)\in I_{f_n}\qquad \forall\,n\geq1.
\]
Using the recursive definition of $s_n(x)$ and the analogue to {\bf (P4)} for $\xi$, we deduce that
\begin{equation}
 s_n(x)=\xi\big(f_m,f_n;s_m(x)\big)\qquad\forall\,x\in I_{f_0}\,,\;\forall\,n,m\in\mathbb N.\label{RECURSIVE.SN}
\end{equation}
Then, to prove that $\{f_n\}_{n\in\mathbb N}$ converges in $\mathbb{FI}(X;W)$, we divide this proof into 2 steps.

\vspace*{5ex}
\end{minipage}\hspace*{1ex}\begin{minipage}[b]{.3\textwidth}
\makeatletter
\def\@captype{figure}
\makeatother
  \centering
\tikzset{every picture/.style={line width=0.75pt}} 
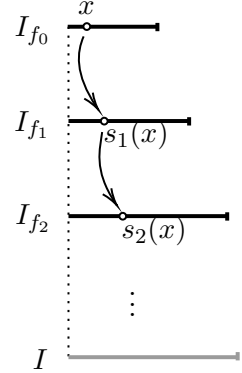
\begin{tikzpicture}[x=0.75pt,y=0.75pt,yscale=-1,xscale=1,scale=.8]
\draw [color={rgb, 255:red, 0; green, 0; blue, 0 }  ,draw opacity=1 ][line width=1.5]    (239.39,48.97) -- (239.39,54.22) ;
\draw [line width=1.5]    (182.83,51.39) -- (239.17,51.39) ;
\draw  [dash pattern={on 0.84pt off 2.51pt}]  (183.33,51.39) -- (183.33,259.89) ;
\draw [line width=1.5]    (183.33,110.89) -- (258.67,110.89) ;
\draw [line width=1.5]    (183.33,170.89) -- (283.28,170.89) ;
\draw [color={rgb, 255:red, 155; green, 155; blue, 155 }  ,draw opacity=1 ][line width=1.5]    (183.33,259.89) -- (289.17,259.89) ;
\draw [color={rgb, 255:red, 0; green, 0; blue, 0 }  ,draw opacity=1 ][line width=1.5]    (259.39,108.47) -- (259.39,113.72) ;
\draw [color={rgb, 255:red, 0; green, 0; blue, 0 }  ,draw opacity=1 ][line width=1.5]    (283.39,168.47) -- (283.39,173.72) ;
\draw [color={rgb, 255:red, 155; green, 155; blue, 155 }  ,draw opacity=1 ][line width=1.5]    (289.99,256.67) -- (289.99,261.92) ;
\draw  [fill={rgb, 255:red, 255; green, 255; blue, 255 }  ,fill opacity=1 ] (192.75,51.29) .. controls (192.75,50.12) and (193.7,49.17) .. (194.88,49.17) .. controls (196.05,49.17) and (197,50.12) .. (197,51.29) .. controls (197,52.47) and (196.05,53.42) .. (194.88,53.42) .. controls (193.7,53.42) and (192.75,52.47) .. (192.75,51.29) -- cycle ;
\draw  [fill={rgb, 255:red, 255; green, 255; blue, 255 }  ,fill opacity=1 ] (203.75,110.96) .. controls (203.75,109.78) and (204.7,108.83) .. (205.88,108.83) .. controls (207.05,108.83) and (208,109.78) .. (208,110.96) .. controls (208,112.13) and (207.05,113.08) .. (205.88,113.08) .. controls (204.7,113.08) and (203.75,112.13) .. (203.75,110.96) -- cycle ;
\draw  [fill={rgb, 255:red, 255; green, 255; blue, 255 }  ,fill opacity=1 ] (215.42,170.96) .. controls (215.42,169.78) and (216.37,168.83) .. (217.54,168.83) .. controls (218.72,168.83) and (219.67,169.78) .. (219.67,170.96) .. controls (219.67,172.13) and (218.72,173.08) .. (217.54,173.08) .. controls (216.37,173.08) and (215.42,172.13) .. (215.42,170.96) -- cycle ;
\draw    (192.89,57.5) .. controls (188.66,67.25) and (186.97,81.12) .. (201.1,102.2) ;
\draw [shift={(202.22,103.83)}, rotate = 235.12] [color={rgb, 255:red, 0; green, 0; blue, 0 }  ][line width=0.75]    (10.93,-3.29) .. controls (6.95,-1.4) and (3.31,-0.3) .. (0,0) .. controls (3.31,0.3) and (6.95,1.4) .. (10.93,3.29)   ;
\draw    (204.56,117.83) .. controls (203.26,128.34) and (198.79,141.49) .. (212.78,162.53) ;
\draw [shift={(213.89,164.17)}, rotate = 235.12] [color={rgb, 255:red, 0; green, 0; blue, 0 }  ][line width=0.75]    (10.93,-3.29) .. controls (6.95,-1.4) and (3.31,-0.3) .. (0,0) .. controls (3.31,0.3) and (6.95,1.4) .. (10.93,3.29)   ;

\draw (160.83,53.83) node  [font=\normalsize]  {$I_{f_{0}}$};
\draw (223.67,221.5) node    {$\vdots $};
\draw (160.5,113.5) node  [font=\normalsize]  {$I_{f_{1}}$};
\draw (161.17,172.5) node  [font=\normalsize]  {$I_{f_{2}}$};
\draw (165.5,260.83) node  [font=\normalsize]  {$I$};
\draw (194.67,39.17) node    {$x$};
\draw (226,119.83) node    {$s_{1}(x)$};
\draw (237.33,180.17) node    {$s_{2}(x)$};
\draw (165.5,300.83) node  [font=\normalsize]  {$\,$};
\end{tikzpicture}

\end{minipage}

\vspace*{-3ex}
\noindent{\em Step 1.} 
We claim that for all $x\in I_{f_0}$, $\big\{f_n(s_n(x))\big\}_{n\in\mathbb N}$ is a Cauchy sequence in $W$ (w.r.t. the Euclidean metric). 

\noindent By \eqref{RECURSIVE.SN} and the analogue of {\bf (P4)} for $\xi$, we have
\[
f_m\circ\xi\big(f_n,f_m;s_n(x)\big)= f_m\big(s_m(x)\big)\qquad\forall\,x\in I_{f_0}\,,\;\,\forall\,n,m\in\mathbb N.\quad
\]
This last equation together to \eqref{CHI.FNFM} imply
\setlength \arraycolsep{2pt}
\begin{eqnarray}
\big|f_n\big(s_n(x)\big)-f_m\big(s_m(x)\big)\big|&=&\big|f_n\big(s_n(x)\big)-f_m\circ\xi\big(f_n,f_m;s_n(x)\big)\big| \nonumber \\
&\leq& \max_{y\,\in\, I_{f_n}}\big|f_n(y)-f_m\circ\xi(f_n,f_m;y)\big|<\epsilon,\label{TRUCO}
\end{eqnarray}
for every $x\in I_{f_0}$ and $n,m\geq n_\epsilon$. This implies that $\big\{f_n(s_n(x))\big\}_{n\in\mathbb N}\subset W$ is a Cauchy sequence and, by completeness of $W$, that the limit $\Lim{n\to\infty}f_n\big(s_n(x)\big)\in W$ exists for every $x\in I_{f_0}$. \\[-1ex]

\noindent{\em Step 2.} Let $I\in\mathbb I(X)$ be such that $I_{f_n}\to I$ w.r.t. the Hausdorff metric. We claim that the sequence $\{f_{n}\}_{n\in\mathbb N}$ converges (w.r.t. $\chi$) to $f:I\to W$ given by
\[
f(y):=\lim_{n\to\infty}f_n\big(s_n\circ\xi_0(y)\big)\qquad\forall\,y\in I,
\]
where $\xi_0:I\to I_{f_0}$ is the increasing map that linearly deforms $I$ onto $I_{f_0}$ (note that $I$ is non-degenerate!). By Step 1, for every $x\in I_{f_0}$ the sequence $\big\{f_n(s_n(x))\big\}_{n\in\mathbb N}$ is convergent in $W$, so that $f$ is well defined. It remains to prove that $f$ is continuous. Using that $\xi_0$ is an homeomorphism from $I$ onto $I_{f_0}$ and taking the limit $m\to\infty$ in \eqref{TRUCO}, we obtain
\begin{eqnarray}
\big|f_n\big(s_n(x)\big)-f\big(\xi_0^{-1}(x)\big)\big|=\lim_{m\to\infty}\big|f_n\big(s_n(x)\big)-f_m\big(s_m\circ\xi_0\circ\xi_0^{-1}(x)\big)\big|<\epsilon,\label{XI0}
\end{eqnarray}
for every $x\in I_{f_0}$ and $n\geq n_\epsilon$. Thus, we deduce that $f_n\circ s_n\circ \xi_0: I\to W$ converges uniformly to $f$ as $n\to\infty$. Since $s_n=\xi(f_{0},f_n;\cdot):I_{f_0}\to I_{f_n}$ is continuous, the composition $f_n\circ s_n\circ\xi_0: I\to W$ is uniformly continuous for all $n\in\mathbb N$, which implies that $f$ is continuous, and hence that $f\in\mathbb{FI}(X;W)$. Finally, note that for any $n\geq n_\epsilon$,
\setlength \arraycolsep{2pt}
\begin{eqnarray*}
\chi(f,f_n)=\max_{y\,\in\, I}\big|f(y)-f_n\circ \xi(f,f_n;y)\big|&=&\max_{x\,\in\, I_{f_0}}\big|f\circ \xi_0^{-1}(x)-f_n\circ \xi\big(f,f_n;\xi_0^{-1}(x)\big)\big|\\
&=&\max_{x\,\in\, I_{f_0}}\big|f\circ \xi_0^{-1}(x)-f_n\circ \xi\big(f_0,f_n;x\big)\big|\\
&\stackrel{\eqref{RECURSIVE.SN}}=& \max_{x\,\in\, I_{f_0}}\big|f\circ \xi_0^{-1}(x)-f_n\big(s_n(x)\big)\big|\stackrel{\eqref{XI0}}<\epsilon,
\end{eqnarray*}
which ends the proof.
\end{proof}

\begin{lemma}\label{DERIVATIVE}
Let $\{f_n\}_{n\in\mathbb N}$ be a Cauchy sequence in $\big(\mathbb{FI}(X;W),\chi\big)$ such that $f_n:I_{f_n}\to W$ is of class $\mathcal C^1$ for every $n\in\mathbb N$, $\{f_n'\}_{n\in\mathbb N}$ is a Cauchy sequence in $\big(\mathbb{FI}(X;\mathbb R),\chi\big)$ and $I_{f_n}$ converges to $I\in\mathbb I(X)$ w.r.t. the Hausdorff metric. Then, $f_n\to f\in\mathbb{FI}(X;W)\,$ and $\,f_n'\to g\in\mathbb{FI}(X;\mathbb R)$ implies that $f:I\to W$ is of class $\mathcal C^1$ and $f'=g$.
\end{lemma}

\begin{proof}
From Lemma \ref{FIX}, the limits $f:I\to W$ and $g:I\to\mathbb R$ exist and are unique in $\mathbb{FI}(X;W)$ and $\mathbb{FI}(X;\mathbb R)$ for the respective sequences. Our goal will be to prove that:
\begin{equation}
f(y)-f(a)=\int_a^{y}g(t)\,dt\qquad\forall\,y\in I:=[a,b].\label{TFC.FINAL}
\end{equation}
Thus, using that $f$ is bounded and continuous on $I$, the conclusion will follow from the Fundamental Theorem of Calculus. 

As in the proof of Lemma \ref{FIX}, for the Cauchy sequence $\{f_n\}_{n\in\mathbb N}$ we consider the sequences $\big\{s_n(x)\big\}_{n\in\mathbb N}$ for $x\in I_{f_0}$, and the increasing linear map $\xi_0:=\xi(f,f_0;\cdot):I\to I_{f_0}$.\\[-2ex]

\noindent First, we claim that if $I_{f_n}=[a_n,b_n]$ for every $n\in\mathbb N$ and $I=[a,b]$, then the limits
\[
\lim_{n\to\infty}\int_{a_n}^{\xi(f,f_n;y)} g\circ\xi(f_n,g;t)\,dt \qquad\text{and}\qquad \lim_{n\to\infty}\int_{a_n}^{\xi(f,f_n;y)} f_n'(t)\,dt
\]
exist for every $y\in I$. First, note that for any $n\in\mathbb N$ and $y\in I$ we have
\setlength \arraycolsep{2pt}
\begin{eqnarray}
\int_{a_n}^{\xi(f,f_n;y)} g\circ\xi(f_n,g;t)\,dt&=&\int_{\xi(f,f_n;a)}^{\xi(f,f_n;y)} g\circ\xi(f_n,g;t)\,dt \nonumber\\
&=&\int_a^y g\circ\xi\big(f,g;t\big)\,\xi'\big(f,f_n;t\big)\,dt \nonumber \\
&=&\int_a^y g(t)\,\xi'\big(f,f_n;t\big)\,dt,\label{EQUALITY.LIMIT}
\end{eqnarray}
where we use $\xi(f,g;\cdot)=\xi(g,f;\cdot)=\id$. Besides, for all $n\in\mathbb N$ the map $\xi'\big(f,f_n;t\big)$ is a constant function such that $\xi'\big(f,f_n;t\big)\to 1$ uniformly as $n\to\infty$, since $I_{f_n}\to I$. Thus, from \eqref{EQUALITY.LIMIT} and since $g(t)\,\xi'\big(f,f_n;t\big)$ converges uniformly to $g(t)$ on $[a,y]$, we deduce that
\begin{equation}
\lim_{n\to\infty}\int_{a_n}^{\xi(f,f_n;y)} g\circ\xi(f_n,g;t)\,dt=\int_a^y g(t)\,dt\qquad \forall\,y\in I,\label{LIMITE1}
\end{equation}
and we conclude that the first limit exists. In the second case, for all $n\in\mathbb N$ and $y\in I$ we have that
\setlength \arraycolsep{2pt}
\begin{eqnarray*}
\int_{a_n}^{\xi(f,f_n;y)} f_n'(t)\,dt= \int_{\xi(f,f_n;a)}^{\xi(f,f_n;y)}f_n'(t)\,dt&=&\int_a^y f_n'\big(\xi(f,f_n;t)\big)\,\xi'(f,f_n;t)\,dt\\
&=& f_n\circ\xi(f,f_n;y)-f_n\circ\xi(f,f_n;a). 
\end{eqnarray*}
Note that from \eqref{RECURSIVE.SN} we obtain
\[
f_n\circ\xi(f,f_n;y) = f_n\circ\xi\big(f_0,f_n;\xi_0(y)\big) = f_n\circ s_n\circ \xi_0(y)\qquad\forall\,n\in\mathbb N\,,\;\forall\,y\in I,
\]
and we conclude that 
\begin{equation}
\lim_{n\to\infty}\int_{a_n}^{\xi(f,f_n;y)} f_n'(t)\,dt=\lim_{n\to\infty}\big(f_n\circ s_n\circ \xi_0(y)-f_n\circ s_n\circ \xi_0(a)\big)=f(y)-f(a), \label{LIMITE2}
\end{equation}
for every $y\in I$. Finally, we claim that the limits \eqref{LIMITE1} and \eqref{LIMITE2} coincide. Indeed, since $\{f_n'\}_{n\in\mathbb N}$ is a Cauchy sequence converging to $g$, for all $\epsilon>0$ there exists $n_\epsilon\in\mathbb N$ such that 
\begin{equation}
\max_{y\,\in\, I_{f_n}}\big|f_n'(y)-g\circ\xi(f_n,g;y)\big|<\frac{\epsilon}{c_N-c_0}\qquad\forall\,n\geq n_\epsilon,\label{MAXFPRIMA}
\end{equation}
where $I_{f_n}\subset X=[c_0,c_N]$ for every $n\in\mathbb N$. Also, note that for any $n\geq n_\epsilon$ and $y\in I$ we have
\setlength \arraycolsep{2pt}
\begin{eqnarray}
\left|\int_{a_n}^{\xi(f,f_n;y)} f_n'(t)\,dt-\int_{a_n}^{\xi(f,f_n;y)} g\circ\xi(f_n,g;t)\,dt\right|&& \nonumber \\
&&\hspace*{-3.5cm}\leq\int_{a_n}^{\xi(f,f_n;y)} \big|f_n'(t)- g\circ \xi(f_n,g;t)\big|\,dt \stackrel{\eqref{MAXFPRIMA}}\leq\epsilon, \label{LESS.EPSILON}
\end{eqnarray}
which implies that \eqref{LIMITE1} and \eqref{LIMITE2} are equals, which proves \eqref{TFC.FINAL} and concludes the proof.
\end{proof}

\begin{remark}\em\label{REMARK1} 
Given $f\in\mathcal{PC}^r_N(X)$, note that $f_i:\overline{X_i(f)}\to X\in\mathbb{FI}(X;X)$ and $f_i^{(j)}:\overline{X_i(f)}\to \mathbb R\in\mathbb{FI}(X;\mathbb R)$ for every $i\in\{1,\ldots,N\}$ and $j\in\{1,\ldots,r\}$.
\end{remark}

Now, we are going to answer the following question: could it be that the collapse of continuity pieces is the only obstacle to the convergence of Cauchy sequences in $\mathcal{PC}^r_N(X)$? Theorem \ref{COLLAPSE} below we answers to this question in the affirmative. 

\begin{definition}\label{DEF.COLLAPSE}
Given a sequence $\underline{s}:=\{f_n\}_{n\in\mathbb N}\subset\mathcal{PC}_N^r(X)$, we define the {\em collapse index} of $\underline{s}$ by
\[
\kappa(\underline{s}):=\#\Big\{1\leq i\leq N\ :\ \lim_{n\to\infty}\big|c_{i-1}^{f_n}-c_i^{f_n}\big|=0\Big\}.
\]
\end{definition}
The collapse index $\kappa(\underline{s})$ measures the number of continuity pieces collapsing to a single point when $n\to\infty$. Also, $0\leq\kappa(\underline{s})\leq N-1$ for every sequence $\underline{s}$. Note that the collapse index of the Cauchy sequence in Example \ref{NOT.CONVERGENT} is 1.

\begin{theorem}\label{COLLAPSE}
Let $\underline{s}$ be a Cauchy sequence in $\big(\mathcal{PC}^r_N(X),\comp^r\!\big)$. Then, $\underline{s}$ is convergent if and only if $\kappa(\underline{s})=0$. In particular, any Cauchy sequence $\{f_n\}_{n\in\mathbb N}\subset \mathcal{PC}^r_N(X)$ such that $c_i^{f_n}=c_i^{f_{n+1}}$ for every $n\geq0$ and $i\in\{1,\ldots,N-1\}$ is convergent.
\end{theorem}

\begin{proof}
The direct implication is trivial: if the Cauchy sequence $\underline{s}$ converges in $\mathcal{PC}_N^r(X)$ then the limit of this sequence has $N$ continuity pieces and, therefore, $\kappa(\underline{s})=0$. Reciprocally, if $\underline{s}=\{f_n\}_{n\in\mathbb N}$ is a Cauchy sequence with $\kappa(\underline{s})=0$, then given $\epsilon>0$ there exists $n_\epsilon\geq n_0$ such that $\comp^r(f_n,f_m)<\epsilon$ for any $n,m\geq n_\epsilon$. This implies also that for any $n,m\geq n_\epsilon$,
\[
\big|c_i^{f_n}-c_i^{f_m}\big|<\epsilon\qquad \forall\,i\in\{1,\ldots,N-1\}
\]
and
\[
\max_{x\in \overline{X_i(f_n)}}\Big|[f_n^{(j)}]_i(x)-[f_m^{(j)}]_i\circ\xi_{i}(f_n,f_m;x)\Big|<\epsilon,
\]
for every $i\in\{1,\ldots,N\}$ and $j\in\{0,\ldots,r\}$, where $[f_n^{(j)}]_i$ denotes the continuous extension of $f_n^{(j)}\big|_{X_i(f_n)}$ to set $\overline{X_i(f_n)}$. Thus, for every $i\in\{1,\ldots,N\}$ we have that $\big\{[f_n]_i\big\}_{n\in\mathbb N}$ and $\big\{[f_n^{(j)}]_i\big\}_{n\in\mathbb N}$, $1\leq j\leq r$, are Cauchy sequences in the metric spaces $\mathbb{FI}(X;X)$ and $\mathbb{FI}(X;\mathbb R)$, respectively. 

Given $i\in\{1,\ldots,N-1\}$, it is clear that each $\big\{c_i^{f_n}\big\}_{n\in\mathbb N}$ is Cauchy in $X$, so that there exist $c_i\in X$ such that $c_i^{f_n}\to c_i$. Now, since $\kappa(\underline{s})=0$ we deduce that there exist intervals $I_i\in\mathbb I(X)$ such that $\overline{X_i(f_n)}\to I_i$ w.r.t. the Hausdorff metric. We also have that $I_i=[c_{i-1},c_i]$.

By Lemma \ref{FIX} we conclude that for any $i\in\{1,\ldots,N\}$ and $j\in\{1,\ldots,r\}$ there exists unique maps $g_{i,j}:I_i\to \mathbb R\in\mathbb{FI}(X;\mathbb R)$ and $g_{i,0}: I_i\to X\in\mathbb{FI}(X;X)$ such that
\begin{eqnarray*}
\forall\,\epsilon>0\;\,\exists\,n^*_\epsilon\in\mathbb N\;:\quad \left\{\begin{array}{l}
     \displaystyle |c_i^{f_n}-c_i|<\frac{\epsilon}{2N}\quad \text{and}  \\[2ex]
     \displaystyle \max_{x\in \overline{X_i(f_n)}}\Big|\big[f_n^{(j)}\big]_i(x)-g_{i,j}\circ\xi_{i}(f_n,g_{i,j};x)\Big|<\frac{\epsilon}{2Nr},
\end{array}\right.
\end{eqnarray*}
for every $n\geq n^*_\epsilon$, any $i\in\{1,\ldots,N\}$ and $j\in\{0,\ldots,r\}$. Now, define the map $f\in\mathcal{PC}^r_N(X)$ by
\[
\overline{X_i(f)}=I_i\quad\text{ and }\quad f_i=g_{i,0}\qquad \forall\,i\in\{1,\ldots,N\}.
\]
From Lemma \ref{DERIVATIVE} we have that $f_i^{(j)}=g_{i,0}^{(j)}=g_{i,j}\,$ for every $i\in\{1,\ldots,N\}$ and $j\in\{1,\ldots,r\}$. Thus, for any $n\geq n^*_\epsilon$ we deduce that
\[
\comp^r(f_n,f)=\sum_{i=1}^{N-1}\big|c_i^{f_n}-c_i\big|+\sum_{i=1}^N\sum_{j=0}^r\chi\big([f_n^{(j)}]_i,f_i^{(j)}\big)<\epsilon,
\]
which proofs that the Cauchy sequence $\{f_n\}_{n\in\mathbb N}$ is convergent in $\mathcal{PC}^r_N(X)$.
\end{proof}

\section{Baire property}
\label{sec4}

We know that the open and dense subsets of $\mathcal{PC}_N^r(X)$ provide the natural topological notion of ``almost all''. Genericity is a weaker notion: a property is {\em generic} if it is verified by every piecewise map in a {\em residual} subset of $\mathcal{PC}_N^r(X)$, that is, a subset that contains a countable intersection of dense open subsets, which is also dense in $\mathcal{PC}_N^r(X)$. 

We already noted that the metric space $\big(\mathcal{PC}^r_N(X),\comp^r\big)$ is not complete, so it is not possible to deduce the Baire property directly. However, the good properties of the metric $\comp^r$ will allow us to prove that, in fact, $\mathcal{PC}^r_N(X)$ is a Baire space.

In what follows, denote by $B^r(f;\delta)$ the open ball in $\mathcal{PC}^r_N(X)$ centered at $f$ and of radius $\delta>0$.

\begin{theorem}\label{BAIRE}
$\big(\mathcal{PC}_N^r(X),\comp^r\!\big)$ is a Baire space.
\end{theorem}
\begin{proof}
Let $\{\mathcal D_n\}_{n\in\mathbb N}$ be a collection of dense open subsets of $\mathcal{PC}^r_N(X)$. We claim that if $\mathcal U$ is any non-empty open set of $\mathcal{PC}^r_N(X)$ then $\mathcal D\cap \mathcal U\neq\emptyset$, where
\[
\mathcal D:=\bigcap_{n\in\mathbb N}\mathcal D_n.
\]
Consider any $f_1\in \mathcal U$ and consider $r_1>0$ satisfying $\overline{B^r(f_1;r_1)}\subset \mathcal U$ and 
\begin{equation}
r_1<\frac{\min\!\big\{\big|c_{i-1}^{f_1}-c_i^{f_1}\big|\,:\, 1\leq i\leq N\big\}}{3}.\label{R1}
\end{equation}
Clearly, $\mathcal D_1\cap B^r(f_1;r_1)\neq \emptyset$ is open. Thus, there exist $f_2\in \mathcal D_1\cap B^r(f_1;r_1)$ and $r_2>0$ such that 
\[
\overline{B^r(f_2;r_2)} \subset \mathcal D_1\cap B^r(f_1;r_1)\quad\text{ and }\quad r_2\leq \frac{r_1}{3^2}.
\]
Again, there exist $f_3\in \mathcal D_2\cap B^r(f_2;r_2)$ and $r_3>0$ such that $\overline{B^r(f_3;r_3)} \subset \mathcal D_2\cap B^r(f_2;r_2)$ and $r_3\leq r_1/3^3$. Applying this argument successively, we obtain sequences $\{f_n\}_{n\geq1}$ and $\{r_n\}_{n\geq1}$ such that 
\[
\overline{B^r(f_{n+1};r_{n+1})} \subset \mathcal D_n\cap B^r(f_n;r_n)\quad\text{ and }\quad r_{n+1}\leq \frac{r_1}{3^{n+1}}\qquad \forall\,n\geq1.
\]
By construction, note that $\underline{s}:=\{f_n\}_{n\geq1}$ is a Cauchy sequence and $B^r(f_{n+1};r_{n+1})\subset B^r(f_{n};r_{n})$ for every $n\geq1$. We claim that $\kappa(\underline{s})=0$. Indeed, for every $n\geq2$ and $i\in\{1,\ldots,N\}$ we have that
\begin{equation}
\big|c_{i-1}^{f_n}-c_i^{f_n}\big|\geq \big|c_{i-1}^{f_n}-c_i^{f_1}\big|-\sum_{j=1}^{n-1}\big|c_{i}^{f_j}-c_i^{f_{j+1}}\big|\geq \big|c_{i-1}^{f_n}-c_i^{f_1}\big|-r_1\sum_{j=1}^{n-1}\frac{1}{3^j}.\label{DES1}
\end{equation}
Now, since $f_n\in B^r(f_1;r_1)$ we deduce that $\big|c_i^{f_n}-c_i^{f_1}\big|<r_1$, which implies that
\[
\big|c_{i-1}^{f_1}-c_i^{f_1}\big|<r_1+\big|c_i^{f_n}-c_{i-1}^{f_1}\big|\quad\stackrel{\text{from }\eqref{R1}}{\Rightarrow}\quad r_1<\frac{r_1}{3}+\frac{\big|c_i^{f_n}-c_{i-1}^{f_1}\big|}{3},
\]
and we deduce that $2r_1<\big|c_i^{f_n}-c_{i-1}^{f_1}\big|$. Using this last inequality and \eqref{DES1}, we have that
\[
\big|c_{i-1}^{f_n}-c_i^{f_n}\big| \geq \big|c_{i-1}^{f_n}-c_i^{f_1}\big|-r_1\sum_{j=1}^{n-1}\frac{1}{3^j}>2r_1-\frac{r_1}{2}>0,
\]
which implies that $\kappa(\underline{s})=0$. Finally, from Theorem \ref{COLLAPSE} we conclude that $\underline{s}$ converges to a some $f\in\mathcal{PC}^r_N(X)$ such that
\[
\{f\} = \bigcap_{n\geq2} \overline{B^r(f_n;r_n)} \subset \bigcap_{n\geq1} \mathcal D_n\cap B^r(f_n;r_n)\subset \bigcap_{n\geq1} \mathcal D_n\cap \mathcal U=\mathcal D\cap \mathcal U.
\]
Thus, $\mathcal D\cap \mathcal U\neq\emptyset$ and we conclude that $\mathcal D$ is dense in $\mathcal{PC}_N^r(X)$.
\end{proof}

\section{Application 1: robustness of non-degenerate critical points}\label{sec5}

It this section we follow the notation from Section \ref{sec3}. Given $f\in\mathbb{FI}(X;W)$ such that $f\in\mathcal C^r(I_f,W)$, $r\geq1$, we define the {\em critical point set} of $f$ by
\[
\Crit(f):=\big\{x\in \Int(I_f)\ :\ f'(x)=0 \big\}\cup\partial I_f,
\]
where $\partial I_f$ denotes the set of endpoints of the interval $I_f$. 
If $r\geq2$, we say that $x_0\in \Crit(f)$ is {\em non-degenerate} if $f''(x_0)\neq0$. 
Note that the condition \eqref{SUMASUMA} below implies that each critical point of $f$, if any, is non-degenerate.

\begin{lemma}\label{NDCRITICAL}
Let $f,f_n\in\mathbb{FI}(X;W)$, $n\in\mathbb N$, be of class $\mathcal C^2$. Assume that $\Crit(f)\cap\Int(I_f)=\{x^*\}$, $f_n^{(j)}\stackrel{\chi\;}{\to}f^{(j)}$ for $j\in\{0,1,2\}$ and that 
\begin{equation}
\big|f'(x)\big|+\big|f''(x)\big|>0\qquad\forall\,x\in I_f.\label{SUMASUMA}
\end{equation}
Then, there exist $m\in\mathbb N$ and a sequence $\{x_n^*\}_{n\geq m}$ in $X$ so that $\{x_n^*\}= \Crit(f_n)\cap\Int(I_{f_n})$ for every $n\geq m$, and moreover $x_n^*\to x^*$.
\end{lemma}

\begin{proof}
From \eqref{SUMASUMA} we deduce that $f''(x^*)\neq 0$ and, from Inverse Function Theorem, there exists an open interval $J\subset \Int(I_f)$ such that $x^*\in J$ and $f'$ is invertible in $J$. Now, given that $f_n^{(j)}\stackrel{\chi\;}{\to}f^{(j)}$ for $j\in\{0,1,2\}$, the maps $f_n\circ\xi(f,f_n;\cdot)$, $f_n'\circ\xi(f,f_n;\cdot)$ and $f''_n\circ\xi(f,f_n;\cdot)$ converge uniformly to $f$, $f'$ and $f''$, respectively. Thus, there exists $m\geq0$ such that for each $n\geq m$, the map $f_n'\circ\xi(f,f_n;\cdot)$ is invertible in $J$ and the equation
\[
f_n'\circ\xi(f,f_n;x)=0\qquad\text{with $x\in J$,}
\]
has a unique solution $x_n\in J$. 
Now, given that $f_n'\stackrel{\chi\;}{\to}f'$ and that $f'_n\circ \xi(f,f_n;x_n)=0$ for $n\geq m$, we deduce 
\[
\big|f'(x_n)\big|=\big|f'(x_n)-f'_n\circ \xi(f,f_n;x_n)\big|\leq \chi(f',f'_n)\qquad\forall\,n\geq m,
\]
which implies that $f'(x_n)\to 0$ as $n\to\infty$. Also, since $f'$ is invertible in $J$ and $\big[f'|_J\big]^{-1}$ is continuous at $0$, we conclude
\begin{equation}
\lim_{n\to\infty}x_n=\lim_{n\to\infty}\big[f'|_J\big]^{-1}\big(f'(x_n)\big)=\big[f'|_J\big]^{-1}(0)=x^*.\label{LIMXN}
\end{equation}
For every $n\geq m$, let $x^*_n:=\xi(f,f_n;x_n)\in \Int(I_{f_n})$ denote the unique critical point of $f_n$. Thus, given $\epsilon>0$, from \eqref{LIMXN} we get that there exists $n_1\geq m$ such that $|x^*-x_n|<\epsilon/2$ for every $n\geq n_1$. Also, from $f_n\stackrel{\chi\;}{\to}f$ there exists $n_2\geq m$ such that $\chi(f_n,f)<\epsilon/2$ for all $n\geq n_2$, which implies, from the property {\bf (P5)} of $\xi$, that
\[
\big|\xi(f,f_n;x_n)-x_n\big|<\frac{\epsilon}{2}\qquad\forall\,n\geq n_2.
\]
Finally, using this last inequality and \eqref{LIMXN} we have that
\[
|x^*_n-x^*|\leq \big|\xi(f,f_n;x_n)-x_n\big|+|x_n-x^*|<\epsilon,
\]
for every $n\geq\max\{n_1,n_2\}$. This completes the proof.
\end{proof}

\begin{theorem}
Let $f\in\mathbb{FI}(X;W)$ be of class $\mathcal C^2$ and such that $\Crit(f)\cap \Int(I_f)=\{x_0\}$. Also, assume that $f$ satisfies \eqref{SUMASUMA}. Then, there exists $\epsilon>0$ such that $1\leq \#\big(\!\Crit(g)\cap\Int(I_g)\big)\leq 3$ for every $g\in\mathbb{FI}(X;W)$ satisfying
\[
\chi(f,g)+\chi(f',g')+\chi(f'',g'')<\epsilon.
\]
\end{theorem}

\begin{proof}
From \eqref{SUMASUMA} we deduce that $x_0$ is a non-degenerate critical point. Thus, from Inverse Function Theorem there exists an open interval $J\subset I_f$ such that $x_0\in J$ and the restriction of $f'$ to $J$ is invertible. Moreover, it is possible to choose $J$ small enough so that $f''(x)\neq0$ for every $x\in \overline{J}$, and such that the endpoints of $I_f$ do not lie in $\overline{J}$. Now, we claim that there exists $\eta_0>0$ such that, for every $g\in \mathbb{FI}(X;W)$ of class $\mathcal C^2$ satisfying $\chi(f,g)+\chi(f',g')+\chi(f'',g'')<\eta_0$, the equation
\begin{equation}
g'\circ\xi(f,g;x)=0\label{GGG}
\end{equation}
has a solution in $J$. Indeed, assume by contradiction that for every $n\geq1$ there exists $g_n\in\mathbb{FI}(X;W)$ of class $\mathcal C^2$ such that
\[
\chi(f,g_n)+\chi(f',g_n')+\chi(f'',g_n'')<\frac{1}{n}
\]
and such that the equation \eqref{GGG} has no solution for $g_n'$. Thus, we have that $g_n^{(j)}\stackrel{\chi\;}{\to} f^{(j)}$ for $j\in\{0,1,2\}$ and, from Lemma \ref{NDCRITICAL}, there exist $m\in\mathbb N$ and a sequence $\{x_n^*\}_{n\geq m}$ such that $\{x_n^*\}=\Crit(f_n)\cap\Int(I_{f_n})$ for every $n\geq m$ and such that $x_n^*\to x^*$, which is our desired contradiction. Moreover, since $f'$ is invertible on $J$, we can choose $\eta_0$ so that such equation \eqref{GGG} has a unique solution for every map $g\in \mathbb{FI}(X;W)$ of class $\mathcal C^2$ satisfying $\chi(f,g)+\chi(f',g')+\chi(f'',g'')<\eta_0$.

First, assuming that $I_f=[a,b]$, we suppose that $f'(a)\neq 0\neq f'(b)$. Thus, we can consider $\epsilon>0$ such that
\[
\epsilon<\frac{1}{2}\min\!\left\{\eta_0\,, \min_{x\in I_f\setminus J}\big|f'(x)\big|\,,\min_{x\in \overline{J}}\big|f''(x)\big|\right\}\!.
\]
Let $g\in\mathbb{FI}(X;W)$ be such that $\chi(f,g)+\chi(f',g')+\chi(f'',g'')<\epsilon$. Let's prove that $g\circ\xi(f,g;\cdot)$ has a unique internal critical point. Indeed, note that
\[
2\epsilon<\big|f'(x)\big|\leq \big|f'(x)-g'\circ\xi(f,g;x)\big|+\big|g'\circ\xi(f,g;x)\big|<\epsilon+\big|g'\circ\xi(f,g;x)\big|\quad\forall\,x\in I_f\setminus J,
\]
which implies that $g\circ\xi(f,g;\cdot)$ has no critical points in $I_f\setminus J$, except endpoints $a$ and $b$. Also, since $\epsilon<\eta_0$ we have that $g\circ\xi(f,g;x)$ has a unique critical point in $J$. Thus, we deduce that $\#\big(\!\Crit(g)\cap\Int(I_g)\big)=1$ for every $g$ of class $\mathcal C^2$ such that $\chi(f,g)+\chi(f',g')+\chi(f'',g'')<\epsilon$.\\[-2ex]

Now, we assume that $f'(a)=0$ and $f'(b)\neq 0$. From \eqref{SUMASUMA}, there exists $\delta>0$ such that
\[
[a,a+\delta]\cap \overline{J}=\emptyset\qquad\text{and}\qquad f''(x)\neq0 \quad\forall\,x\in[a,a+\delta].
\]
We claim that there exists $\eta_0'>0$ such that, for every $g$ of class $\mathcal C^2$ satisfying $\chi(f,g)+\chi(f',g')+\chi(f'',g'')<\eta_0'$, the equation
\begin{equation}
g'\circ\xi(f,g;x)=0\qquad\text{with $x\in[a,a+\delta)$,}\label{DELTA.EQ}
\end{equation}
has at most one solution. Define the map $\tilde f:[2a-b,b]\to W$ by
\[
\tilde f(x):=\left\{\begin{array}{ll}
    f(x) &\text{ if $x\in[a,b]$,}  \\
    f(2a-x) & \text{ if $x\in[2a-b,a)$.} 
\end{array}\right.
\]
It is very easy to verify that $\tilde f$ is differentiable in $x=a$ and that $a$ is a non-degenerate internal critical point for $\tilde f$ (moreover, every $\mathcal C^2$-extension of $f$ to interval $[2a-b,b]$ satisfies that $a$ is a non-degenerate internal critical point for such extension). Using a similar argument to deduce the existence of $\eta_0$, we can prove that there exists $\eta_1>0$ such that the equation 
\begin{equation*}
g'\circ\xi(f,g;x)=0\qquad\text{with $x\in(a-\delta,a+\delta)$,}
\end{equation*}
has a unique solution for maps $g$ of class $\mathcal C^2$ satisfying that $\chi(\tilde f,g)+\chi(\tilde f',g')+\chi(\tilde f'',g'')<\eta_1$. Since $\tilde f|_{[a,b]}=f$, we deduce that the existence of $\eta_0'>0$ such that the equation \eqref{DELTA.EQ} has at most one solution for maps $g$ of class $\mathcal C^2$ satisfying $\chi(f,g)+\chi(f',g')+\chi(f'',g'')<\eta_0'$.

Now, define $U:=[a,a+\delta)\cup J$ and consider $\epsilon>0$ such that 
\[
\epsilon<\frac{1}{2}\min\!\left\{\eta_0\,,\eta_0'\,, \min_{x\in I_f\setminus U}\big|f'(x)\big|\,,\min_{x\in \overline{U}}\big|f''(x)\big|\right\}\!.
\]
Let $g\in\mathbb{FI}(X;W)$ be such that $\chi(f,g)+\chi(f',g')+\chi(f'',g'')<\epsilon$. Let's prove that $g\circ\xi(f,g;\cdot)$ has at least one critical point and at most 2 critical points. Using analogous arguments to those given in the previous case, we conclude that $g\circ\xi(f,g;\cdot)$ has a unique critical point in $J$, and has no critical points in $I_f\setminus\overline{U}$. Also, as $\epsilon<\eta_0'$ we have that $g\circ\xi(f,g;\cdot)$ has at most one critical point in $[a,a+\delta)$, therefore, $1\leq \#\big(\!\Crit(g)\cap\Int(I_g)\big)\leq2$ for every $g$ of class $\mathcal C^2$ satisfying $\chi(f,g)+\chi(f',g')+\chi(f'',g'')<\epsilon$.\\[-2ex]

The case $f'(a)\neq0$ and $f'(b)=0$ is argued in a similar way, concluding the same as above. Finally, for the case $f'(a)=f'(b)=0$ is necessary to apply the 2 previous arguments simultaneously, concluding that $1\leq \#\big(\!\Crit(g)\cap\Int(I_g)\big)\leq3$ for every $g$ of class $\mathcal C^2$ satisfying $\chi(f,g)+\chi(f',g')+\chi(f'',g'')<\epsilon$, where $U$ and $\epsilon$ must be chosen appropriately.
\end{proof}

\begin{corollary}\label{NDCRITICAL.2}
Let $f\in\mathbb{FI}(X;W)$ be of class $\mathcal C^2$ such that $\#\Crit(f)<\infty$ and such that $f$ satisfies \eqref{SUMASUMA}. Then, there exists $\epsilon>0$ such that $\#\big(\!\Crit(f)\cap\Int(I_f)\big)\leq \#\big(\!\Crit(g)\cap\Int(I_g)\big)\leq \#\big(\!\Crit(f)\cap\Int(I_f)\big)+2$ for every $g\in\mathbb{FI}(X;W)$ satisfying $\chi(f,g)+\chi(f',g')+\chi(f'',g'')<\epsilon$. If in addition $f'(x)\neq0$ for $x\in\partial I_f$, then there exists $\epsilon>0$ such that $\#\Crit(g)=\#\Crit(f)$ for every $g\in\mathbb{FI}(X;W)$ satisfying $\chi(f,g)+\chi(f',g')+\chi(f'',g'')<\epsilon$.
\end{corollary}

For PC${}^r$-maps, $r\geq1$, we say that $x_0\in X$ is a {\em critical point} of $f\in\mathcal{PC}_N^r(X)$ if either $x_0\in \Delta_f\cup\{c_0,c_N\}$, or $x_0\in X\setminus(\Delta_f\cup\{c_0,c_N\})$ and $f'(x_0)=0$. Let $\Crit(f)$ denote the set of critical points of $f$. Note that all local extrema occur at critical points. We say that a critical point $x_0$ of $f\in\mathcal{PC}_N^r(X)$, $r\geq2$, is {\em non-degenerate} if
\[
f_i''(x_0)\neq 0\qquad\Leftrightarrow \qquad x_0\in\overline{X_i(f)}.
\]
Note that if $x_0\in X\setminus(\Delta_f\cup\{c_0,c_N\})$ and $f'(x_0)=0$, then the condition $f''(x_0)\neq 0$ implies that there is a neighborhood $J$ of $x_0$ such that $f'(x)\neq 0$ for each $x$ in $J\setminus\{x_0\}$. From this we deduce that the non-degenerate critical points inside continuity pieces are isolated. For this reason, we define the {\em internal critical points set} of $f\in\mathcal{PC}^r_N(X)$, $r\geq1$, by
\[
\Crit_{\text{int}}(f):=\big\{x\in X\setminus(\Delta_f\cup\{c_0,c_N\}) \ :\ f'(x)=0\big\}.
\]
The elements of $\Crit_{\text{int}}(f)$ are called {\em internal critical points} of $f$. It is clear that $\Crit(f)=\Crit_{\text{int}}(f)\cup\big(\Delta_f\cup\{c_0,c_N\}\big)$. The following result is deduced directly from Corollary \ref{NDCRITICAL.2}.

\begin{corollary}\label{PC.CRIT.POINT}
Let $r\geq2$ and $f\in\mathcal{PC}_N^r(X)$ be such that $\#\Crit(f)<\infty$, and assume that for every $i\in\{1,\ldots,N\}$ and $x\in\overline{X_i(f)}$,
\[
\big|f_i'(x)\big|+\big|f_i''(x)\big|>0.
\]
Then, there exists $\epsilon>0$ such that {\em $\#\Crit_{\text{int}}(f)\leq \#\Crit_{\text{int}}(g)\leq \#\Crit_{\text{int}}(f)+2N$} for every $g\in B^r(f;\epsilon)$. If in addition $f_s'(c_j^f)\neq0$ for every $s\in\{j,j+1\}$ and $j\in\{0,\ldots,N-1\}$, then there exists $\epsilon>0$ such that $\#\Crit(f)=\#\Crit(g)$ for every $g\in B^r(f;\epsilon)$.
\end{corollary}

\section{Application 2: Genericity of PC-maps that admit an invariant Borel probability measure}\label{sec6}

B. Pires proved in \cite{P16} that PC${}^0$-maps {\em without connections} (see Definition \ref{NO.CONECCTIONS} below) admit an invariant Borel probability measure. PC${}^0$-maps without connections have been used to establish other important results in the theory --such as the spectral decomposition of the attractor of piecewise contracting interval maps (see \cite{CCG20})-- so it is important to know if this condition is typical.

\begin{definition}\label{NO.CONECCTIONS}\em
We say that $f\in\mathcal{PC}_N^r(X)$ {\em has no connections} if
\[
D_f:=\big\{f_s(c_j^f)\ :\ s\in\{j,j+1\}\,,\;0\leq j\leq N-1\big\}\subset \bigcap_{j=0}^\infty f^{-j}(X\setminus\Delta_f).
\]
\end{definition}
\noindent In other words, PC${}^0$-maps without connections satisfy that the orbit of one-sided limits at boundaries of continuity pieces do not intersect the set of boundaries of continuity pieces again. With this condition, we can state the result that Pires proved:

\begin{theorem}[(B. Pires, 2016 - \cite{P16})]\label{PIRES}
If $f\in\mathcal{PC}_N^0(X)$ is a PC-map without connections, then $f$ admits an invariant Borel probability measure.
\end{theorem}

\noindent Thus, showing the genericity of PC${}^0$-maps without connections allows us to conclude the genericity of PC${}^0$-maps that admit invariant Borel probability measures. As we will see, we will need to restrict the collection to the subspace of piecewise Lipschitz maps:
\setlength \arraycolsep{2pt}
\begin{eqnarray*}
\mathcal{P\!L}_N(X)&:=&\big\{f\in\mathcal{PC}_N^0(X) \ :\ \exists\,\lambda>0\,\text{ s.t. }\, |f(x)-f(y)|\leq\lambda\,|x-y|\\
&&\hspace*{6.4cm}\forall\,x,y\in X_i(f)\,,\; 1\leq i\leq N\big\}.
\end{eqnarray*}
Thus, for $f\in\mathcal{P\!L}_N(X)$ we define the {\em Lipschitz constant} of $f$ by
\[
\lambda_f:=\max_{1\leq i\leq N}\left[\sup_{\begin{subarray}{c}\scriptscriptstyle x,y\in X_i(f)\\ x\neq y\end{subarray}}\frac{|f(x)-f(y)|}{|x-y|}\right]\!.
\]

\begin{lemma}\label{MISMA.PIEZA3}
Let $\ell\geq1$, $f,g\in\mathcal{P\!L}_N(X)$ and $\epsilon>0$ be such that for every $k\in\{0,\ldots,\ell-1\}$, $f^k(D_f)\cap\Delta_f=\emptyset$, $g^k(D_g)\cap\Delta_g=\emptyset$ and
\[
\comp^r(f,g)< \epsilon < \frac{\min\!\big\{|c-f^k(d)|\ :\ c\in \Delta_f\,,\;d\in D_f \text{ and } k\in\{0,\ldots,\ell-1\}\big\}}{
2\big(1+\lambda_f+\lambda_f^2+\cdots+\lambda_f^{\ell-1}\big)}.
\]
Then, given $i\in\{1,\ldots,N-1\}$, for every $k\in\{0,\ldots,\ell-1\}$ we have 
\begin{equation}
\max_{s\in\{i,i+1\}}\Big\{\big|f^{k}\big(f_s(c_i^f)\big)-\xi\big(g,f;g^{k}(g_s(c_i^g))\big)\big|\Big\}<2\epsilon\sum_{j=0}^{k} \lambda_f^j,\label{2ESUM}
\end{equation}
and for every $j\in\{0,\ldots,k\}$ and $s\in\{i,i+1\}$, the points $f^{j}\big(f_s(c_i^f)\big)$ and $\xi\big(g,f;g^{j}(g_s(c_i^g))\big)$ lie in same continuity piece of $f$. Also,
\begin{equation}
\max_{s\in\{i,i+1\}}\Big\{\big|f^{k+1}\big(f_s(c_i^f)\big)-\xi\big(g,f;g^{k+1}(g_s(c_i^g))\big)\big|\Big\}<2\epsilon\sum_{j=0}^{k+1} \lambda_f^j.\label{2ESUM2}
\end{equation}
\end{lemma}

\begin{proof}
Fix $i\in\{1,\ldots,N-1\}$ and $\ell\geq 2$ (if $\ell=1$, the proof is reduced to step 1 of the following induction). We proceed by induction on $k\in\{0,\ldots,\ell-1\}$. Note that
\begin{eqnarray}
\big|f_s(c_i^f)-\xi\big(g,f;g_s(c_i^g)\big)\big|&\leq& \big|f_s(c_i^f)-g_s\circ \xi\big(f,g;c_i^f\big)\big| + \big|g_s(c_i^g)-\xi\big(g,f;g_s(c_i^g)\big)\big|\nonumber \\
&<&2\epsilon\leq 2\epsilon\sum_{j=0}^{\ell-1} \lambda_f^j,\label{FF1}
\end{eqnarray}
where $s\in\{i,i+1\}$. Since 
\[
2\epsilon\sum_{j=0}^{\ell-1} \lambda_f^j < \min\!\big\{|c-f^j(d)|\ :\ c\in \Delta_f\,,\;d\in D_f \text{ and } j\in\{0,\ldots,\ell-1\}\big\},
\]
from \eqref{FF1} we conclude that $f_s(c_i^f)$ and $\xi\big(g,f;g_s(c_i^g)\big)$ belong to same continuity piece of $f$ \big(note that $f_s(c_i^f)$ belongs to $X_{j}(f)$ if and only if $g_s(c_i^g)$ belongs to $X_{j}(g)$, for some $j\in\{1,\ldots,N\}$\big). Also, we have 
\setlength \arraycolsep{2pt}
\begin{eqnarray*}
\big|f\big(f_s(c_i^f)\big)-\xi\big(g,f;g(g_s(c_i^g))\big)\big|&\leq& \big|f\big(f_s(c_i^f)\big)-f\circ \xi\big(g,f;g_s(c_i^g)\big)\big|\\ 
&& \hspace*{1cm}+\big|f\circ \xi\big(g,f;g_s(c_i^g)\big)-g\big(g_s(c_i^g)\big)\big|\\
&&\hspace*{2.2cm}+\big|g\big(g_s(c_i^g)\big)-\xi\big(g,f;g(g_s(c_i^g))\big)\big|\\
&\leq&\lambda_f \big|f_s(c_i^f)- \xi\big(g,f;g_s(c_i^g)\big)\big|+2\epsilon<2\epsilon(\lambda_f+1),
\end{eqnarray*}
which prove the case $k=0$. Now, we assume that for $k$ (such that $k\leq\ell-2$) are valid the following statements:
\begin{itemize}
    \item \eqref{2ESUM} and \eqref{2ESUM2} are satisfied for every $j\in\{0,\ldots,k\}$ and $s\in\{i,i+1\}$;
    \item For every $j\in\{0,\ldots,k\}$ and $s\in\{i,i+1\}$, the points $f^{j}\big(f_s(c_i^f)\big)$ and $\xi\big(g,f;g^{j}(g_s(c_i^g))\big)$ belong to same continuity piece of $f$.
\end{itemize}
From \eqref{2ESUM2}, for every $s\in\{i,i+1\}$,
\[
\big|f^{k+1}\big(f_s(c_i^f)\big)-\xi\big(g,f;g^{k+1}(g_s(c_i^g))\big)\big|\leq 2\epsilon\sum_{j=0}^{k+1} \lambda_f^j\leq 2\epsilon\sum_{j=0}^{\ell-1} \lambda_f^j,
\]
and we deduce that $f^{k+1}\big(f_s(c_i^f)\big)$ and $\xi\big(g,f;g^{k+1}(g_s(c_i^g))\big)$ belong to same continuity piece of $f$ \big(note that, $f^{k+1}\big(f_s(c_i^f)\big)$ belongs to $X_{j}(f)$ if and only if $g^{k+1}\big(g_s(c_i^g)\big)$ belongs to $X_{j}(g)$, for some $j\in\{1,\ldots,N\}$\big). Also, are valid the following inequalities:
\setlength \arraycolsep{2pt}
\begin{eqnarray*}
\big|f^{k+2}\big(f_s(c_i^f)\big)-\xi\big(g,f;g^{k+2}(g_s(c_i^g))\big)\big|&\leq& \big|f^{k+2}\big(f_s(c_i^f)\big)-g^{k+2}\big(g_s(c_i^g)\big)\big|\\
&&\hspace*{.3cm}+\big|g^{k+2}\big(g_s(c_i^g)\big)-\xi\big(g,f;g^{k+2}(g_s(c_i^g))\big)\big| \nonumber \\
&\leq& \big|f^{k+2}\big(f_s(c_i^f)\big)-f\circ\xi\big(g,f;g^{k+1}(g_s(c_i^g))\big)\big| \nonumber \\
&&\hspace*{-.8cm}+\big|f\circ\xi\big(g,f;g^{k+1}(g_s(c_i^g))\big)-g^{k+2}\big(g_s(c_i^g)\big)\big|+\epsilon \nonumber \\
&<& 2\epsilon\lambda_f\sum_{j=0}^{k+1} \lambda_f^j+2\epsilon\;=\;2\epsilon\sum_{j=0}^{k+2} \lambda_f^j,
\end{eqnarray*}
where this last inequality is deduced from \eqref{2ESUM2} and since $f^{k+1}\big(f_s(c_i^f)\big)$ and $\xi\big(g,f;g^{k+1}(g_s(c_i^g))\big)$ belong to same continuity piece of $f$. Thus, we prove the result for $k+1$.
\end{proof}

\begin{lemma}\label{NO.INTERSECTA}
Let $n\geq1$, $f,g\in\mathcal{PC}^r_N(X)$ and $\epsilon>0$ be such that for every $k\in\{0,\ldots,n-1\}$, $f^k(D_f)\cap\Delta_f=\emptyset$ and
\begin{equation}
\comp^r(f,g)< \epsilon < \frac{\min\!\big\{|c-f^k(d)|\ :\ c\in \Delta_f\,,\;d\in D_f \text{ and } k\in\{0,\ldots,n-1\}\big\}}{
4\big(1+\lambda_f+\lambda_f^2+\cdots+\lambda_f^{n-1}\big)}.\label{EPSILON.CHICO}
\end{equation}
Then, for every $k\in\{0,\ldots,n-1\}$ we have that $g^k(D_g)\cap\Delta_g=\emptyset$. In particular, for any $n\geq1$ the set
\begin{equation*}
\mathcal F_n:=\big\{f\in\mathcal{PC}_N^r(X)\ :\ f^{j}(D_f)\cap\Delta_f=\emptyset\quad\forall\,j\in\{0,\ldots,n-1\}\big\}
\end{equation*}
is open in $\mathcal{PC}_N^r(X)$.
\end{lemma}
\begin{proof}
Let $n\geq 2$. We proceed by induction on $k\in\{0,\ldots,n-1\}$. For every $i,j\in\{1,\ldots,N-1\}$ and $s\in\{i,i+1\}$, we have that
\setlength \arraycolsep{2pt}
\begin{eqnarray*}
4\epsilon<\big|c^f_j-f_s(c_i^f)\big|&\leq& \big|c^f_j-c^{g}_j\big|+\big|f_s(c_i^f)-\xi\big(g,f;g_s(c_i^g)\big)\big|\\
&&\hspace*{2.5cm}+\big|\xi\big(g,f;g_s(c_i^g)\big)-g_s(c_i^g)\big|+\big|c^{g}_j-g_s(c_i^g)\big|.
\end{eqnarray*}
From hypothesis $\comp^r(f,g)<\epsilon$ and Lemma \ref{MISMA.PIEZA3}, we deduce that $0<\big|c^{g}_j-g_s(c_i^g)\big|$. This prove the result for $k=0$. Now, we assume that 
\[
D_g\cap\Delta_g=g(D_g)\cap\Delta_g=\cdots=g^k(D_g)\cap\Delta_g=\emptyset
\]
for $k\leq n-2$. Then, from Lemma \ref{MISMA.PIEZA3} (here, the role of $\ell-1$ in such result is done by $k$), for every $i,j\in\{1,\ldots,N-1\}$ and $s\in\{i,i+1\}$ we have that $f^{k}\big(f_s(c_i^f)\big)$ and $\xi\big(g,f;g^{k}(g_s(c_i^g))\big)$ belong to same continuity piece of $f$, and is valid the inequality \eqref{2ESUM2}. Thus, we have that  
\setlength \arraycolsep{2pt}
\begin{eqnarray*}
4\epsilon\sum_{l=0}^{n-1} \lambda_f^l<\big|c^f_j-f^{k+1}\big(f_s(c_i^f)\big)\big|&\leq& \big|c^f_j-c^{g}_j\big|+\big|f^{k+1}\big(f_s(c_i^f)\big)-g^{k+1}\big(g_s(c_i^g)\big)\big|\\[-2ex]
&&\hspace*{3.85cm}+\big|c^{g}_j-g^{k+1}\big(g_s(c_i^g)\big)\big|\\
&&\hspace{-.8cm}\leq \big|c^f_j-c^{g}_j\big|+\big|f^{k+1}\big(f_s(c_i^f)\big)-\xi\big(g,f;g^{k+1}(g_s(c_i^g))\big)\big|\\
&&\hspace*{1.1cm}+\big|\xi\big(g,f;g^{k+1}(g_s(c_i^g))\big)-g^{k+1}\big(g_s(c_i^g)\big)\big|\\
&& \hspace*{3.88cm}+\big|c^{g}_j-g^{k+1}\big(g_s(c_i^g)\big)\big|\\
&&\hspace{-.8cm}<2\epsilon+2\epsilon\sum_{l=0}^{k+1} \lambda_f^l+\big|c^{g}_j-g^{k+1}\big(g_s(c_i^g)\big)\big|\\
&&\hspace{-.8cm}\leq 2\epsilon+2\epsilon\sum_{l=0}^{n-1} \lambda_f^l+\big|c^{g}_j-g^{k+1}\big(g_s(c_i^g)\big)\big|,
\end{eqnarray*}
which prove the result for $k+1$. Thus, we ending the proof of the first part of Lemma. \\[-2ex]

Finally, note that if $n\geq 1$ and $f\in\mathcal F_n$, its clear that $B^r(f;\epsilon)\subset\mathcal F_n$ for every $\epsilon$ satisfying the condition \eqref{EPSILON.CHICO}, therefore, $\mathcal F_n$ is a open set.
\end{proof}

Note that we have proven that the ``no connections'' condition is $\comp^0$-robust in the subspace $\mathcal{P\!L}_N(X)$. The genericity of this condition is a consequence of the following result:\\

\begin{lemma}\label{FN.LEMMA}
For all $n\geq1$, the set $\mathcal F_n$ from Lemma \ref{NO.INTERSECTA} is dense in $\mathcal{P\!L}_N(X)$. 
\end{lemma}

\begin{proof}
Fix $f\in\mathcal{P\!L}_N(X)$, $n\geq1$ and $\epsilon>0$. If $f\in\mathcal F_n$, then it is done. If $f\notin\mathcal F_n$ we have that
\[
\bigcup_{j=0}^{n-1}\big(D_f\cap f^{-j}(\Delta_f)\big)\neq\emptyset.
\]
Define the subsets of $D_f$ given by
\setlength \arraycolsep{2pt}
\begin{eqnarray*}
D_1&:=&\big\{d\in D_f\setminus\Delta_f\ :\ \exists\,k\in\{1,\ldots,n-1\}\text{ such that } f^k(d)\in\Delta_f\big\}\neq\emptyset,\\
D_2&:=&\big\{d\in D_f\ :\ f^k(d)\notin\Delta_f\quad\forall\,k\in\{0,\ldots,n-1\}\big\}.
\end{eqnarray*}
Note that for every $d\in D_1$ there exists $n(d)\in\{0,\ldots,n-1\}$ such that $f^{n(d)}(d)\in\Delta_f$ and such that $f^k(d)\notin\Delta_f$ for every $k\in\{1,\ldots,n(d)-1\}$. If $d\in D_2$ we define $n(d)=n$. Using these notations, we can consider an increasing homeomorphism $h:X\to X$ such that
\begin{enumerate}
\item[$1^\times$.]\quad $\max\limits_{x\in X}|h(x)-x|<\frac{\epsilon}{4N(1+\lambda_f)}$;
\item[$2^\times$.]\quad $h(x)=x\,$ for every $x\in D^*$, where
\[
D^*:=\bigcup_{d\in D_1\cup D_2}\!\bigcup_{j=0}^{\scriptscriptstyle n(d)-1}\big\{f^{j}(d)\big\};
\]
\item[$3^\times$.]\quad $h^{-1}(\Delta_f)\subset X\setminus \big(\Delta_f\cup D^*\big)$.
\end{enumerate}
Then, if we consider $g:=f\circ h\in\mathcal{PC}_N^0(X)$, the continuity pieces of $g$ are the pre-images by $h$ of the continuity pieces of $f$. Moreover, from $1^\times$ follows that
\[
\comp^0(f,g)\leq\sum_{i=1}^{N-1}\big|h(c_i^g)-c_i^g\big|+ \sum_{i=1}^N\,\max_{x\in\overline{X_i(g)}}\lambda_f\big|h (x)-\xi(g,f;x)\big|<\epsilon.
\]
Now, note that $g_s(c_j^g)=f_s(c_j^f)\in D_f$ for all $j\in\{1,\ldots,N-1\}$ and $s\in\{j,j+1\}$ (in particular, $D_g=D_f$). Also, from $2^\times$ we have that $f^k(d)=g^k(d)$ for all $d\in D_f$ and $k\in\{0,\ldots,n(d)-1\}$. Thus, from $3^\times$ we deduce
\[
\big|g^k(d)-c_\ell^g\big|=\left\{\begin{array}{ll}
    \big|c-h^{-1}(c_\ell^f)\big|>0 &\quad\text{ if $g^k(d)=c\in\Delta_f$,}  \\
    \big|f^k(d)-h^{-1}(c_\ell^f)\big|>0 & \quad\text{ if $k\leq n(d)-1$,}
\end{array}\right.
\]
for every $\ell\in\{1,\ldots,N-1\}$. This implies that $g^j(D_g)\cap\Delta_g=\emptyset$ for all $j\in\{0,\ldots,n-1\}$, which proves the lemma. 
\end{proof}

\begin{corollary}\label{PC.GENE}
There exists a residual set $\mathcal F\subset \mathcal{P\!L}_N(X)$ such that every map of $\mathcal F$ admits an invariant Borel probability measure. 
\end{corollary}

\begin{proof}
Note that
\setlength \arraycolsep{2pt} 
\begin{eqnarray*}
\mathcal F:=\bigcap_{n\in\mathbb N} \mathcal F_n &=& \big\{f\in\mathcal{PC}^r_N(X)\ :\ f^j(D_f)\cap\Delta_f=\emptyset\quad\forall\,j\in\mathbb N\big\}\\
&=& \big\{f\in\mathcal{PC}^r_N(X)\ :\ \text{$f$ has no connections}\big\}
\end{eqnarray*}
Thus, from Theorem \ref{BAIRE} and Lemma \ref{FN.LEMMA} we deduce that $\mathcal F$ is dense in $\mathcal{P\!L}_N(X)$ and from Theorem \ref{PIRES} we obtain the genericity of maps in $\mathcal{P\!L}_N(X)$ that admit an invariant Borel probability measure.
\end{proof}

\noindent {\bf Acknowledgments:} AC was supported by CIMFAV Project CIDI0403-2021, Universidad de Valparaíso, Chile.

\end{document}